\documentclass[10pt,reqno]{amsart}

\usepackage{amssymb}
\usepackage{amsfonts}
\usepackage{amsthm}
\usepackage{amsmath}

\numberwithin{equation}{section}
\allowdisplaybreaks

\newtheorem{theorem}{Theorem}[section]
\newtheorem{proposition}[theorem]{Proposition}

\newtheorem{lemma}[theorem]{Lemma}
\newtheorem{corollary}[theorem]{Corollary}

\theoremstyle{definition}

\theoremstyle{remark}

\newcommand{\R}{\mathbb{R}}

\newcommand{\N}{\mathbb{N}}
\newcommand{\Z}{\mathbb{Z}}
\newcommand{\C}{\mathbb{C}}

\newcommand{\eps}{\varepsilon}

\newcommand{\scriptA}{\mathcal{A}}
\newcommand{\scriptB}{\mathcal{B}}
\newcommand{\scriptJ}{\mathcal{J}}
\newcommand{\scriptP}{\mathcal{P}}
\newcommand{\scriptW}{\mathcal{W}}

\newcommand{\qtq}[1]{\quad\text{#1}\quad}
\newcommand{\ctc}[1]{\,\,\text{#1}\,\,}

\DeclareMathOperator*{\ch}{ch}
\DeclareMathOperator*{\supp}{supp}
\DeclareMathOperator*{\diam}{diam}
\DeclareMathOperator*{\dist}{dist}

\begin{document}
\title{Uniform $L^p$-improving for weighted averages on curves}
\author{Betsy Stovall}
\address{Department of Mathematics, 480 Lincoln Dr., Madison, WI 53706--1325}
\email{stovall@math.wisc.edu}

\begin{abstract}
We define variable parameter analogues of the affine arclength measure on curves and prove near-optimal $L^p$-improving estimates for associated multilinear generalized Radon transforms.  Some of our results are new even in the convolution case.
\end{abstract}

\maketitle

%%%%%%%%%%%%%%%%%%%%%%%%%%%%%%%%%%%%%%
%%%%%%%%%%%%%%%%%%%%%%%%%%%%%%%%%%%%%%
%%%%%%%%%%%%%%%%%%%%%%%%%%%%%%%%%%%%%%

\section{Introduction}

%%%%%%%%%%%%%%%%%%%%%%%%%%%%%%%%%%%%%%
%%%%%%%%%%%%%%%%%%%%%%%%%%%%%%%%%%%%%%
%%%%%%%%%%%%%%%%%%%%%%%%%%%%%%%%%%%%%%

In this article we consider weighted versions of multilinear generalized Radon transforms of the form 
\begin{equation} \label{E:unweighted}
M_0(f_1,\ldots,f_k) := \int_{\R^d} \prod_{i=1}^k f_i\circ \pi_i(x)\, a(x)\, dx,
\end{equation}
where $a$ is a continuous cutoff function and the $\pi_i:\R^d \to \R^{d-1}$ are smooth submersions.  

In \cite{TW, BSrevista}, near endpoint estimates of the form 
\begin{equation} \label{E:bound unweighted}
|M_0(f_1,\ldots,f_k)| \leq C \prod_{i=1}^k \|f_i\|_{L^{p_i}(\R^{d-1})},
\end{equation}
with $C=C(\pi_1,\ldots,\pi_k,p_1,\ldots,p_k)$, were established for $M_0$ under the assumption that the $\pi_i$ satisfy a certain finite type condition on the support of $a$.  In particular, it was found that the exponents on the right on \eqref{E:bound unweighted} depend on this `type.'  These results are nearly sharp in the sense that if the type of the $\pi_i$ degenerates anywhere on the set where $a \neq 0$, then the corresponding near endpoint estimates also fail.  It is not, however, known in general what happens when the type degenerates at some point where $a \neq 0$ (for instance, on the boundary of the support) or the rate at which the constants in \eqref{E:bound unweighted} blow up as the type degenerates.  

Our goal is to quantify and counteract the failure of \eqref{E:bound unweighted} in such situations by replacing $M_0$ by an appropriately weighted operator, for which we will establish near-optimal Lebesgue space bounds.  The exponents (though not the implicit constants) in these bounds will be independent of the choice of $\pi_1,\ldots,\pi_k$ and the cutoff function $a$.  Further, the weights we employ transform naturally under changes of coordinates, so they may reasonably be viewed as generalizations of the affine arclength measure on curves in $\R^d$.  A number of recent articles (such as \cite{BOS1, DLW, DM11, DSjfa, DW, DrM2, Oberlin, OberlinJFA10, OberlinMich, Sjolin, BSjfa}) have been devoted to establishing uniform estimates for operators weighted by affine arclength measure, and these results provide much of the motivation for this article.

\subsection{A motivating example}

Stating the main results of this article, or even the results of \cite{TW, BSrevista} requires some notation, so we postpone this until the next section.  By way of background and motivation, we will spend the remainder of the introduction describing a concrete case about which much is known, and which provides the inspiration for the more general operators considered in this article.  Let $\gamma:\R \to \R^d$ be a smooth curve and $a$ a continuous cutoff function.  Consider the operator
$$
T_0f(x) := \int_\R f(x-\gamma(t))\, a(t)\, dt, \qquad f \in C^0_0(\R^d).
$$
By duality, $T_0:L^{p}(\R^d) \to L^{q}(\R^d)$ if and only if for all $f \in L^p(\R^d)$ and $g \in L^q(\R^d)$, 
$$
\left|\int_{\R^d} \int_\R f(x-\gamma(t)) g(x) \, a(t)\, dt\right| \leq C(\gamma,p,q) \|f\|_{L^p(\R^d)}\|g\|_{L^{q'}(\R^d)};
$$
this may be compared with \eqref{E:bound unweighted}.  

The curve $\gamma$ is said to be of type (at most) $N$ when $\det(\gamma'(t),\ldots,\gamma^{(d)}(t))$ vanishes to order at most $N$ at any point.  The results of \cite{DSjfa2} imply that if $\gamma$ is of type $N$ on the support of $a$, $\|T_0\|_{L^p \to L^q} <\infty$ if $(p^{-1},q^{-1})$ lies in the trapezoid with vertices
\begin{equation} \label{E:vertices}
(0,0), \quad (1,1), \quad (p_N^{-1}, q_N^{-1}) :=  (\tfrac{d}{N+\frac{d(d+1)}2}, \tfrac {d-1}{N+\frac{d(d+1)}2}), \quad (1-q_N^{-1},1-p_N^{-1}).
\end{equation}
(The non-endpoint result was due to Tao--Wright in \cite{TW}.)  
Further, if $N$ is the maximal type of $T_0$ on $\{t:a(t) \neq 0\}$, this is sharp.  If $\gamma$ is not of finite type, $T_0$ satisfies no $L^p(\R^d) \to L^q(\R^d)$ estimates off the line $\{p=q\}$.

It was first noticed in \cite{Sjolin} and \cite{DM85} that affine, as opposed to Euclidean, arclength has a uniformizing effect on the bounds for convolution and Fourier restriction operators associated to possibly degenerate curves.  It is now known that if $\gamma$ is a polynomial curve, convolution with affine arclength measure on $\gamma$, which is the operator
$$
Tf(x) := \int_\R f(x-\gamma(t))\, |\det(\gamma'(t),\ldots,\gamma^{(d)}(t))|^{\frac2{d(d+1)}}\, dt,
$$
maps $L^p(\R^d)$ boundedly into $L^q(\R^d)$ if and only if (provided $T \not \equiv 0$) $(p^{-1},q^{-1})$ lies on the line segment joining $(p_0^{-1},q_0^{-1}), (1-q_0^{-1},1-p_0^{-1})$, with $p_0,q_0$ defined as above (\cite{Oberlin, DLW, BSjfa}).  Further, the operator norms established in \cite{Oberlin, DLW, BSjfa} depend only on the degree of the polynomial; for this, it is crucial that the affine arclength transforms nicely under reparametrizations and affine transformations.  Further investigations have been carried out by Oberlin and Dendrinos--Stovall in the non-polynomial case in \cite{OberlinJFA10, DSjfa2}.  The above mentioned results are essentially optimal, both in terms of the exponents involved and in terms of pointwise estimates on the weight, \cite{OberlinMich} (cf.\ Proposition~\ref{P:optimality}).  Analogous results are also known for the restricted X-ray transform, \cite{DSjfa, DSjfa2}.  There have also been a number of recent articles aimed at establishing uniform estimates for Fourier restriction to curves with affine arclength measure, for instance \cite{BOS1, DM11, DW, BSajm}.  

Our goal in this article is to address the gap between the general results of \cite{TW, BSrevista} and the type-independent results of \cite{DLW, DSjfa,Oberlin,  BSjfa} by introducing a generalization of the affine arclength measure, well-suited to \eqref{E:unweighted}.  We will also prove near-endpoint bounds for the weighted operator and, in particular, will generalize the results of \cite{TW, BSrevista} to the case when the $\pi_i$ completely fail to be of finite type on the support of $a$.  Some of our results are new even in the translation invariant case.

%%%%%%%%%%%%%%%%%%%%%%%%%%%%%%%%%%%%%%
%%%%%%%%%%%%%%%%%%%%%%%%%%%%%%%%%%%%%%
%%%%%%%%%%%%%%%%%%%%%%%%%%%%%%%%%%%%%%

\section{Basic notions and statements of the main results}

%%%%%%%%%%%%%%%%%%%%%%%%%%%%%%%%%%%%%%
%%%%%%%%%%%%%%%%%%%%%%%%%%%%%%%%%%%%%%
%%%%%%%%%%%%%%%%%%%%%%%%%%%%%%%%%%%%%%

\subsection*{Notation}
Throughout the article, we will use the now-standard notation $A \lesssim B$ to mean that $A \leq C B$ for some innocuous implicit constant $C$.  The value of this constant will be allowed to change from line to line.  The meaning of `innocuous' will be specified at the beginning of most sections, though in this section it will be specified \textit{in situ} and in the next, it does not arise.  Additionally, $A \gtrsim B$ if $B \lesssim A$, and $A \sim B$ if $A \lesssim B$ and $B \lesssim A$.  We denote the nonnegative integers by $\Z_0$.   If $\ell$ is any integer, $\delta$ is an $\ell$-tuple of real numbers, and $\beta \in \Z_0^\ell$ is a multiindex, we denote by $\delta^\beta$ the quantity $\delta_1^{\beta_1} \cdot \ldots \cdot \delta_\ell^{\beta_\ell}$.  

We will also use some less-standard notation.  We consider the partial order $\preceq$ on $\Z_0^k$ defined by $b_1 \preceq b_2$ if $b_1^i \leq b_2^i$, $1 \leq i \leq k$.  We say $b_1 \prec b_2$ if at least one of these inequalities is strict.  If $\scriptB \subseteq \Z_0^k$, is any set, we define a polytope
$$
\scriptP(\scriptB) := \ch \bigcup_{b \in \scriptB} ([0,\infty)^k + \{b\}),
$$
where `$\ch$' denotes the convex hull.  

Fix a dimension $d$ and an integer $k \geq 2$; $k$ may exceed $d$.  We will consider vector fields $X_1,\ldots,X_k$, defined and smooth on the closure of an open set $U$.  A word $w$ is an element of $\scriptW := \bigcup_{n=1}^\infty\{1,\ldots,k\}^n$.  To each word is associated a vector field $X_w$, defined recursively by $X_{(i)} := X_i$, $1 \leq i \leq k$ and $X_{(w,i)} := [X_w,X_i]$, for $w \in \scriptW$ and $1 \leq i \leq k$.  The degree of $w \in \scriptW$ is the $k$-tuple, $\deg w$, whose $i$-th entry is the number of occurrences of $i$ in $w$.  

All brackets of such vector fields lie in the span of the $X_w$:  if $w,w' \in \scriptW$, 
\begin{equation} \label{E:Jacobi}
[X_w,X_{w'}] = \sum_{\deg \tilde w = \deg w + \deg w'} C_{w,w'}^{\tilde w} X_{\tilde w},
\end{equation}
where $C_{w,w'}^{\tilde w}$ is an integer.  Indeed, by the Jacobi identity,
$$
[X_w,[X_{w'},X_i]] = [[X_w,X_{w'}],X_i] - [X_{(w,i)},X_{w'}],
$$
and so \eqref{E:Jacobi} is easily obtained by inducting on $\|\deg w'\|_{\ell^1}$.  (This was observed in \cite{Hormander}.)  We note that for each $b \in \N^k$, there are only finitely many words $w$ with $\deg w = b$, so the sum in \eqref{E:Jacobi} is finite.  

If $I = (w_1,\ldots,w_d)$ is a $d$-tuple of words, we define $\deg I := \sum_{i=1}^d \deg w_i$ and 
$$
\lambda_I := \det(X_{w_1},\ldots,X_{w_d}).
$$
The Newton polytope of the vector fields $X_1,\ldots,X_k$ at the point $x_0 \in U$ is defined to be
$$
\scriptP_{x_0} := \scriptP(\{\deg I : \ctc{$I$ is a $d$-tuple of words satisfying} \lambda_I(x_0) \neq 0\}),
$$
and we define the Newton polytope of a set $A \subseteq U$ to be
$$
\scriptP_A := \ch(\bigcup_{x \in A} \scriptP_x).
$$

The H\"ormander condition is the statement that $\scriptP_{x_0} \neq \emptyset$ for each $x_0 \in U$.  When the $X_i$ are nonvanishing vector fields tangent to the fibers of the $\pi_i$, this is the finite type hypothesis in \cite{TW, BSrevista}.  

\subsection*{Results}

Let $U \subseteq \R^d$ be an open set and let $\pi_1,\ldots,\pi_k : \overline U \to \R^{d-1}$ be smooth submersions (i.e.\ having surjective differentials).  Letting $\star$ denote the composition of the Hodge-star operator, which maps $(d-1)$-forms to one-forms, with the natural identification of one-forms with vectors via the Euclidean metric, we define vector fields
\begin{equation} \label{E:def X}
X_j := \star(d\pi_j^1 \wedge \cdots \wedge d\pi_j^{d-1}), \qquad 1 \leq j \leq k.
\end{equation}
Let $a$ be a continuous function with compact support contained in $U$. 

Fix a $d$-tuple of words $I_0 = (w_1,\ldots,w_d)$ and define the generalized affine arclength
\begin{equation} \label{E:def rho}
\rho = \rho_{I_0} := |\det(X_{w_1},\ldots,X_{w_d})|^{\frac1{|\deg I_0|_1-1}},
\end{equation}
where $|b|_1$ denotes the $\ell_1$ norm.  Define a $k$-linear form $M:[C^0(\R^d)]^k\to \C$ by 
\begin{equation} \label{E:def M}
M(f_1,\ldots,f_k) := \int_{\R^d} \prod_{j=1}^k f_j \circ \pi_j(x) \, \rho(x)\, a(x)\, dx.
\end{equation}
For $b \in \R^k$ with $|b|_1 >1$, define 
\begin{equation} \label{E:q(b)}
\mathbf q(b) := \tfrac b{|b|_1-1}.
\end{equation}
It is easy to check that $\mathbf q$ equals its own inverse. The following is our main theorem.

\begin{theorem} \label{T:main}
Assume that $\deg I_0$ is an extreme point of $\scriptP_{\supp a}$.  Then for all $\mathbf{p} \in [1,\infty]^k$ satisfying $(p_1^{-1},\ldots,p_k^{-1}) \preceq \mathbf q(b)$ and $p_j^{-1}<q_j(b)$ when $(\deg I_0)_j \neq 0$, we have the estimate
\begin{equation} \label{E:main bound}
|M(f_1,\ldots,f_k)| \lesssim \prod_{j=1}^k \|f_j\|_{L^{p_j}(\R^{d-1})},
\end{equation}
for all continuous $f_1,\ldots,f_k$.  The implicit constant depends on the $\pi_j$, $a$, $\mathbf p$ and $b_0$, but not on the $f_j$.  Thus $M$ extends to a bounded $k$-linear form on $\prod_{j=1}^k L^{p_j}(\R^{d-1})$.  
\end{theorem}

The extremality hypothesis seems natural by analogy with the translation invariant case; it also leads to certain invariants of the weight, as we will discuss below.  However, we ultimately prove a more general result, Theorem~\ref{T:not extreme}, which does not require extremality.  (We postpone stating the latter because it requires more notation.)

With the given weight, the above theorem is nearly sharp.  Indeed, under the hypotheses and notation above, we have the following.
     
\begin{proposition}\label{P:optimality}
Let $\mu$ be a nonnegative Borel measure whose support is contained in $U$, and assume that the bound
\begin{equation} \label{E:M with mu}
M_\mu(\chi_{E_1},\ldots,\chi_{E_k}) := \int_{\R^d} \prod_{j=1}^k \chi_{E_j} \circ \pi_j \, d\mu \leq A(\mu) \prod_{j=1}^k |E_j|^{\frac1{p_j}}
\end{equation}
holds for all Borel sets $E_1,\ldots,E_k \subseteq\R^{d-1}$ and some constant $A(\mu) < \infty$.  If $\mu \not \equiv 0$, $(p_1,\ldots,p_k) \in [1,\infty]^k$.  If $\sum_j p_j^{-1} > 1$, let $b_p := \mathbf q(p_1^{-1},\ldots,p_k^{-1})$.  Then $\mu(\{x : b_p \notin \scriptP_x\}) = 0$.  If in addition, $b_p$ is an extreme point of $\scriptP_{\supp \mu}$, $\mu$ is absolutely continuous with respect to Lebesgue measure, and its Radon--Nikodym derivative satisfies
\begin{equation} \label{E:mu lesssim rho}
\tfrac{d\mu}{dx} \lesssim A(\mu) \sum_{\deg I = b_p} |\lambda_I|^{\frac1{|b_p|_1-1}}.
\end{equation}
The implicit constant in \eqref{E:mu lesssim rho} may be chosen to depend only on $d,p$; $A(\mu)$ has the same value in \eqref{E:M with mu} and \eqref{E:mu lesssim rho}.  
\end{proposition}

In the translation invariant case, a similar result is due to D.~Oberlin in \cite{OberlinMich} (cf.\ \cite{DSjfa} for the restricted X-ray transform).  The final statement in the proposition only applies in the endpoint case, which is not otherwise addressed in this article.  The endpoint version of Theorem~\ref{T:main} is known to fail without further assumptions on the $X_i$ than made here, as can be seen by considering the example of convolution with affine arclength on $\gamma(t) = (t,e^{-1/t}\sin(\tfrac1{t^k}))$, $t > 0$, for $k$ sufficiently large.  (This example is due to Sj\"olin in \cite{Sjolin}.)

The proofs of Theorem~\ref{T:main} and Proposition~\ref{P:optimality} will rely on a more general result about smooth vector fields $X_1,\ldots,X_k$ on $\R^d$.  To state this result, we need some additional terminology.  

Let $J \in \{1,\ldots,k\}^d$.  We define $\deg J$ to be the $k$-tuple whose $i$-th entry is the number of occurrences of $i$ in $J$.  If $\alpha \in \Z_0^d$ is a multi-index, we define $\deg_J \alpha$ to be the $k$-tuple whose $i$-th entry is $\sum_{\ell:J_{\ell}=i}\alpha_\ell$.  We define
\begin{equation} \label{E:def Psi}
\Psi^J_{x_0}(t_1,\ldots,t_d) := \exp(t_d X_{J_d}) \circ \cdots \circ \exp(t_1 X_{J_1})(x_0).
\end{equation}
We define another polytope, 
$$
\begin{aligned}
\widetilde \scriptP_{x_0} := \scriptP(\{\deg J + &\deg_J\alpha: \: J \in \{1,\ldots,k\}^d \ctc{and} \alpha \in (\Z_0)^d \\  &  \ctc{satisfy}  \partial_t^\alpha \det D \Psi_{x_0}^J(0) \neq 0\}).
\end{aligned}
$$

\begin{proposition} \label{P:equivalence} For each $x_0 \in U$, $\widetilde \scriptP_{x_0} = \scriptP_{x_0}$.  Furthermore, for each extreme point $b_0$ of $\scriptP_{x_0}$, 
\begin{equation} \label{E:main equivalence}
\sum_{\deg I = b_0} |\lambda_I(x_0)| \sim \sum_{J \in \{1,\ldots,k\}^d} \sum_{\underset{\deg J + \deg_J\alpha = b_0}{\alpha \in (\Z_0)^d:}} |\partial_t^\alpha \det D\Psi_{x_0}^J(0)|.
\end{equation}
The implicit constants may be taken to depend only on $d$ and $b_0$, and in particular, may be chosen to be independent of the $X_i$.  
\end{proposition}

\subsection*{Examples}  We take a moment to discuss a few concrete cases where these results apply.

\emph{The translation-invariant case.}  Let $\gamma:\R \to \R^d$ be a smooth map and for $(t,x) \in \R^{1+d}$, define $\pi_1(t,x) = x$, $\pi_2(t,x) = x-\gamma(t)$.  Thus the unweighted operator $M_0$ in \eqref{E:unweighted} is essentially convolution with Euclidean arclength measure on $\gamma$, paired with a test function.  

Using the definition above, $X_1 = \partial_t$, $X_2 = \partial_t + \gamma' \cdot \nabla_x$.  If $w$ is any word of length $n \geq 2$ and if the first two letters of $w$ are 1 and 2, $X_w(t,x) = \gamma^{(n)}(t)$.  If $d \geq 2$, the H\"ormander condition is equivalent to the statement that the torsion of $\gamma$ does not vanish to infinite order at any point.  We note in particular that
\begin{align*}
|\det(X_1,X_2,X_{(1,2)},\ldots,X_{(1,\ldots,1,2)})| &= |\det(X_1,X_2,X_{(2,1)},\ldots,X_{(2,\ldots,2,1)})| \\ & = |\det(\gamma',\ldots,\gamma^{(d)})|,
\end{align*}
and if $U$ is any open set, the only extreme points of $\scriptP_U$ (unless $\scriptP_U$ is empty) are 
$$
\bigl(\tfrac{d(d-1)}2+1, d\bigr), \qquad \bigl(d, \tfrac{d(d-1)}2+1 \bigr).
$$
Thus the affine arclength in this case is defined in the usual way:
$$
\rho(t,x) = |\det(\gamma'(t), \ldots,\gamma^{(d)}(t))|^{\frac2{d(d+1)}}.
$$

By Theorem~\ref{T:main}, for any smooth $\gamma:\R \to \R^d$, and any continuous cutoff function $a$, the convolution operator
$$
Tf(x) = \int f(x-\gamma(t)) \, |\det(\gamma'(t),\ldots,\gamma^{(d)}(t)|^{\frac2{d(d+1)}}\, a(t)\, dt
$$
maps $L^p(\R^d)$ into $L^q(\R^d)$ whenever $(p^{-1},q^{-1})$ lies in the interior of the trapezoid with vertices as in \eqref{E:vertices} in the case $N=0$.  For general smooth curves this result is new, but, as mentioned in the introduction, even stronger results are known in some special cases.  

\emph{Restricted X-ray transforms.}  Let $\gamma:\R \to \R^{d-1}$ be a smooth map and for $(s,t,x) \in \R^{1+1+d-1}$, define $\pi_1(s,t,x) := (t,x)$, $\pi_2(s,t,x) := (s,x-s\gamma(t))$.  Then the operator $M_0$ in \eqref{E:unweighted} is the restricted X-ray transform
$$
Xf(t,x) = \int_\R f(s,x-s\gamma(t))\, a(s,t)\, ds,
$$
paired with a test function.  Using the above definition, 
$$
X_1 = \partial_s, \qquad X_2 = \partial_t + s\gamma'(t)\cdot \nabla_x.
$$
If $d \geq 3$, the only $d+1$-tuples of words $(w_1,\ldots,w_{d+1})$ with $\det(X_{w_1},\ldots,X_{w_{d+1}}) \not \equiv 0$ are, after reordering, those satisfying
$$
w_1 = 1, \qquad w_2 = 2, \qquad w_i = (1,2,\cdots,2), \quad 3 \leq i \leq d+1.
$$
Thus the only extreme point of the Newton polytope is $(d,1+\frac{d(d-1)}2)$, and 
$$
\rho(s,t,x) = |\det(\gamma'(t),\ldots,\gamma^{(d-1)}(t))|^{\frac2{d(d+1)}},
$$
which is a power of the usual affine arclength.  Theorem~\ref{T:main} thus gives a partial generalization of the results of \cite{DSjfa}, wherein a sharp strong type bound for the X-ray transform restricted to polynomial curves with affine arclength was established.

\emph{Generalized Loomis--Whitney.}  Let $\pi_1,\ldots,\pi_d:\R^d \to \R^{d-1}$ be smooth submersions.  The point $(1,\ldots,1)$ is always extreme or in the exterior of the Newton polytope, so for $\eps > 0$
$$
\left|\int_{\R^d} \prod_{i=1}^d f_i \circ \pi_i(x) \, |\det(X_1,\ldots,X_d)(x)|^{\frac1{d-1}}\, a(x) \, dx\right| \lesssim \prod_{i=1}^d \|f_i\|_{L^{d-1+\eps}(\R^{d-1})},
$$
with the implicit constant depending on the $\pi_i$ and $\eps$.
In the case when the $X_i$ do span at every point of the support of $a$, the endpoint estimate was proved in \cite{BCW}.  (The classical Loomis--Whitney inequality is the endpoint estimate when the $\pi_i$ are linear and $a \equiv 1$.)  
          
\subsection*{Outline} In Section~\ref{S:invariants}, we show that the weights we employ satisfy certain natural invariants; this makes them reasonable generalizations of the usual affine arclength measure.  In Section~\ref{S:equivalence}, we prove Proposition~\ref{P:equivalence} by employing the results of \cite{Street} and using a compactness argument.  We also use a combinatorial lemma, whose proof is postponed to the appendix.  In Section~\ref{S:optimality}, we prove the optimality result, Proposition~\ref{P:optimality}.  Finally, in Section~\ref{S:main}, we prove a more general result, Theorem~\ref{T:not extreme}, which implies Theorem~\ref{T:main}.  Our techniques for the proof of the main theorem are essentially those of \cite{ChRl, TW, BSrevista}, with some modifications to handle the potential failure of the H\"ormander condition.

\subsection*{Acknowledgements}  The idea for this project came from a conversation with Michael Christ, and Terence Tao provided many valuable suggestions during the early stages of this work.  The author is very grateful for these discussions.  She would also like to thank the anonymous referee for a detailed and enormously helpful report.  This project was supported in part by NSF DMS-0902667 and 1266336.  

%%%%%%%%%%%%%%%%%%%%%%%%%%%%%%%%%%%%%%
%%%%%%%%%%%%%%%%%%%%%%%%%%%%%%%%%%%%%%
%%%%%%%%%%%%%%%%%%%%%%%%%%%%%%%%%%%%%%

\section{Invariants of the affine arclengths} \label{S:invariants}

%%%%%%%%%%%%%%%%%%%%%%%%%%%%%%%%%%%%%%
%%%%%%%%%%%%%%%%%%%%%%%%%%%%%%%%%%%%%%
%%%%%%%%%%%%%%%%%%%%%%%%%%%%%%%%%%%%%%

Let $U$, $\pi_1,\ldots,\pi_k$, and $X_1,\ldots,X_k$ be as defined above.  For $1 \leq j \leq k$, let $V_j:=\pi_j(U)$.  Fix a $d$-tuple of words $I_0$, and assume that $b_0 := \deg I_0$ is minimal in the sense that if $\deg I' \prec \deg I_0$, $\lambda_I \equiv 0$.  (This minimality is essential.)  Define $\rho$ as in \eqref{E:def rho}.  

\begin{proposition} \label{P:diffeo}
Let $F:U \to \R^d$ and $G_j:V_j \to \R^{d-1}$, $1 \leq j \leq k$, be smooth maps.  Define $\tilde\pi_j := G_j \circ\pi_j \circ F$, $1 \leq j \leq k$, and let $\widetilde X_j$, $\tilde \rho$ be defined as in \eqref{E:def X}, \eqref{E:def rho}, with tildes inserted.  Then 
\begin{equation} \label{E:diffeo}
\tilde \rho = \bigl(\prod_{j=1}^k |(\det DG_j) \circ \pi_j|^{\mathbf q_j(b_0)}\bigr) |\det DF| \, \rho \circ F,
\end{equation}
where $\mathbf q$ is defined as in \eqref{E:q(b)}.
\end{proposition}

In the notation above, let $a$ be a continuous, compactly supported function with $\supp a \subseteq U$, and define
$$
\widetilde M(f_1,\ldots,f_k) := \int_U \prod_{j=1}^k f_j \circ \tilde\pi_j(x)\, \tilde \rho(x) \, a\circ F(x)\, dx.
$$
Proposition~\ref{P:diffeo} implies that if each $G_j$ is equal to the identity and $F$ is one-to-one, then
$$
\widetilde M(f_1,\ldots,f_k) = M(f_1,\ldots,f_k).
$$
If we simply assume that $F$ and all of the $G_j$'s are one-to-one, the proposition implies that for $(p_1^{-1},\ldots,p_k^{-1}) := \mathbf q(b_0)$, 
$$
\sup_{f_1,\ldots,f_k \not \equiv 0} \frac{\widetilde M(f_1,\ldots,f_k)}{\prod_{j=1}^k \|f_j\|_{L^{p_j}(\R^{d-1})}} = \sup_{f_1,\ldots,f_k \not \equiv 0} \frac{M(f_1,\ldots,f_k)}{\prod_{j=1}^k \|f_j\|_{L^{p_j}(\R^{d-1})}}.
$$
We stress, however, that our theorem covers only the non-endpoint cases satisfying $(p_1^{-1},\ldots,p_k^{-1}) \neq \mathbf q(b_0)$ and $b_0$ extreme, so it is not known that either side is finite except in certain cases (cf.\ \cite{BCW, DLW, DSjfa, Oberlin, BSjfa}).  

If we fix $j$, we may consider the family of curves $\gamma_j^{\underline{x}}(t) := \pi_j(\underline x,t)$.  For any smooth one-to-one function $\phi:\R \to \R$, $(\underline x,t) \mapsto (\underline x,\phi(t))$ is also smooth and one-to-one and has Jacobian determinant $\phi'(t)$.  Thus we obtain the following.

\begin{corollary}
The generalized affine arclength defines a parametrization-invariant measure on each of the curves $\gamma_j^{\underline x} = \pi_j(\underline x,t)$.
\end{corollary}

\begin{proof}[Proof of Proposition~\ref{P:diffeo}]  
We will prove the proposition first when the $G_j$ are equal to the identity and then when $F$ is.  The general case follows by taking compositions.  

In the first case, it suffices by simple approximation arguments to prove the identity when $\det DF \neq 0$.  In this case, careful computations reveal that
$$
\widetilde X_j = (\det DF) F^* X_j,
$$
where $F^*$ is the pullback by $F$, given by 
\begin{equation} \label{E:pushforward}
F^*X := (DF)^{-1}X \circ F.
\end{equation}

For $1 \leq i \leq k$, let $Y_i = F^* X_i$.  Then by naturality of the Lie bracket, $Y_w = F^* X_w$, $w \in \scriptW$.  By induction (with base case $w=(j)$), the coordinate expression for the Lie bracket ($[X,X'] = X(X')-X'(X)$), and the product rule, for each $w \in \scriptW$,
\begin{equation} \label{E:Yw plus error}
\widetilde X_w = (\det DF)^{|\deg w|_1} Y_w + \sum_{\deg w' \prec \deg w} f_{w,w'} Y_{w'},
\end{equation}
where the $f_{w,w'}$ are smooth functions.  

By \eqref{E:Yw plus error}, \eqref{E:pushforward}, and our minimality assumption,
\begin{align*}
&\det(\widetilde X_{w_1},\ldots,\widetilde X_{w_d}) \\
& \qquad = (\det DF)^{|b_0|_1}\det(Y_{w_1},\ldots,Y_{w_d}) + \sum_{b' \prec b_0 } \sum_{\deg I' = b'} f_{I,I'} \det(Y_{w_1'},\ldots,Y_{w_d'})\\
& \qquad = (\det DF)^{|b_0|_1-1}\det(X_{w_1},\ldots,X_{w_d})\circ F + 0.
\end{align*}
This completes the proof in the first case.  

In the second case, when $F$ is the identity, it is easy to compute $\widetilde X_j = [(\det DG_j) \circ \pi_j] X_j$, and it can be shown using the product rule and minimality of $b_0$ (as above) that 
$$
\det(\widetilde X_{w_1},\ldots,\widetilde X_{w_d}) = \prod_{j=1}^k [(\det DG_j)\circ \pi_j]^{b_0^j} \det(X_{w_1},\ldots,X_{w_d}),
$$
which implies \eqref{E:diffeo}.  
\end{proof}

%%%%%%%%%%%%%%%%%%%%%%%%%%%%%%%%%%%%%%
%%%%%%%%%%%%%%%%%%%%%%%%%%%%%%%%%%%%%%
%%%%%%%%%%%%%%%%%%%%%%%%%%%%%%%%%%%%%%

\section{Equivalence of the two polytopes:  The proof of Proposition~\ref{P:equivalence}} \label{S:equivalence}

%%%%%%%%%%%%%%%%%%%%%%%%%%%%%%%%%%%%%%
%%%%%%%%%%%%%%%%%%%%%%%%%%%%%%%%%%%%%%
%%%%%%%%%%%%%%%%%%%%%%%%%%%%%%%%%%%%%%

Fix a point $b_0 \in [0,\infty)^k$.  We say that an object (such as a constant, vector, or set) is admissible if it may be chosen from a finite collection, depending only on $b_0$ and $d$, of such objects.  In particular, all implicit constants in this section will be admissible.  

The proof of Proposition~\ref{P:equivalence} will rely on the following compactness result about polytopes with vertices in $\Z_0^k$.

\begin{proposition}\label{P:polytope prop}
Let $\scriptB \subseteq \Z_0^k$ and assume that $b_0\notin\scriptP(\scriptB)$.  There exist \\
(i)  $\eps > 0$ and $v_0 \in (\eps,1]^k$ such that $v_0 \cdot b_0 + \eps < v_0 \cdot p$ for every $p \in \scriptP(\scriptB)$\\
(ii) a finite set $\scriptA \subseteq \Z_0^k$ such that $b_0 \notin \scriptP(\scriptA)$ and $\scriptP(\scriptB) \subseteq \scriptP(\scriptA)$.\\
Moreover, $\eps,v_0,\scriptA$ are admissible.
\end{proposition}

Note that this proposition also applies when $b_0$ is an extreme point of $\scriptP(\scriptB)$, since in this case $b_0 \notin \scriptP(\scriptB \setminus \{b_0\})$.  

Assuming the validity of Proposition~\ref{P:polytope prop} for now (it will be proved in the Appendix), we devote the remainder of the section to the proof of Proposition~\ref{P:equivalence}.  

We may of course assume that $x_0 = 0$ and that $U$ is a bounded neighborhood of 0.  Furthermore, we may assume that $k > d$ and $X_i = \partial_i$, $1 \leq i \leq d$.  Indeed, if the proposition holds under this assumption, it holds for $\partial_1,\ldots,\partial_d,X_1,\ldots,X_k$, with $k+d$ replacing $k$.  We may then transfer the result back to $X_1,\ldots,X_k$ by restricting to those $b \in [0,\infty)^{k+d}$ with $b^{1} = \cdots = b^{d} = 0$.  By this assumption, $\scriptP_0 \neq \emptyset$, and it suffices to prove that if $b_0$ is an extreme point of $\scriptP_{x_0}$, then \eqref{E:main equivalence} holds, and if $b_0 \notin \scriptP_{x_0}$, then $b_0 \notin \tilde\scriptP_{x_0}$.  

We begin with the case when $b_0$ is an extreme point of $\scriptP_0$.  Fix a neighborhood $V$ of $0$, sufficiently small for later purposes, with $\overline V \subseteq U$.  Choose a $d$-tuple $I_0 = (w_1,\ldots,w_d) \in \scriptW^d$ with $\deg I_0 = b_0$ and 
\begin{equation} \label{E:I=wi}
|\lambda_{I_0}(0)| = \max_{\deg I = b_0} |\lambda_I(0)|.
\end{equation}
(Note that $I_0$ is admissible, since only finitely many $d$-tuples of words give rise to this degree.)  By smoothness of the $X_j$, we may assume that $V$ is so small that 
$$
\tfrac14 |\lambda_{I_0}(0)| \leq \tfrac12 \max_{\deg I = b_0}|\lambda_I(x)| \leq |\lambda_{I_0}(x)| \leq 2|\lambda_{I_0}(0)|, \qquad x \in V.
$$

By Proposition~\ref{P:polytope prop}, we may choose admissible $v_0 = (v_0^1,\ldots,v_0^k) \in (0,1]^k$ and $\eps > 0$ such that $v_0 \cdot b_0 + \eps < v_0 \cdot p$ for every $p \in \scriptP_0\cap\Z_0^k \setminus \{b_0\}$.  

\begin{lemma} \label{L:Ys}
For each $m \geq 1$, there exists $\delta(m) > 0$, depending on $m, b_0,X_1,\ldots,X_k$, such that for all $0 < \delta < \delta(m)$, the map
\begin{equation} \label{E:define Phi delta}
\Phi^\delta(y_1,\ldots,y_d) := \exp(y_1 \delta^{v_0 \cdot \deg w_1}X_{w_1}+\cdots+y_d\delta^{v_0 \cdot w_d} X_{w_d})(0)
\end{equation}
and pullbacks
\begin{equation} \label{E:define Y delta}
Y_j^\delta := (\Phi^\delta)^* \delta^{v_0^j} X_j = (D\Phi^\delta)^{-1} \delta^{v_0^j} X_j \circ \Phi^\delta
\end{equation}
satisfy the following properties:  $\Phi^\delta$ is a diffeomorphism of the unit ball $B(1)$ onto a neighborhood of 0 in $V$,
\begin{gather}
\label{E:det DPhidelta}
|\det D\Phi^\delta(y)| \sim \delta^{v_0 \cdot b_0} |\lambda_{I_0}(0)|, \qquad y \in B(1), \\
\label{E:Ys smooth}
\|Y_j^\delta\|_{C^m(B(1))} \lesssim 1, \qquad 1 \leq j \leq k\\
\label{E:det Ys is 1}
|\det(Y_{w_1}^\delta(y),\ldots,Y^\delta_{w_d}(y))| \sim 1, \qquad y \in B(1).
\end{gather}
\end{lemma}

\begin{proof}
Recall that $\scriptW$ is the set of all words.  Let 
\begin{equation}\label{E:Wi}
\scriptW_0 := \{w \in \scriptW : \deg w \cdot v_0 \leq d\} \:\: \text{and}\:\: \scriptW_1 := \{w \in \scriptW : d < \deg w \cdot v_0 \leq 2d\}.
\end{equation}
Since $v_0$ is an admissible element of $(0,1]^k$, these are admissible, finite sets, and $\scriptW_0$ contains the one-letter words:  $(1),(2),\ldots,(k)$.  Furthermore, $\scriptW_0$ contains $b_0$ since our choice of $v_0$ and assumption that $X_j = \partial_j$, $1 \leq j \leq d$, imply that
$$
v_0 \cdot b_0 \leq v_0 \cdot (1,\ldots,1,0,\ldots,0) = (v_0)_1 + \cdots + (v_0)_d \leq d.
$$

The vector fields $X_w$ are all smooth, $\scriptW_0 \cup \scriptW_1$ is a finite set, and each coefficient of $v_0$ is positive.  Thus for each $M \geq 0$, for all sufficiently small $\delta>0$ and all $w \in \scriptW_0 \cup \scriptW_1$,
\begin{equation} \label{E:deltaXsmooth}
\|\delta^{v_0 \cdot \deg w} X_w\|_{C^0(V)} \leq \tfrac1d\dist(0,\partial V), \qquad \|\delta^{v_0 \cdot\deg w}X_w\|_{C^M(V)} \leq 1.
\end{equation}
Additionally, by our choice of $v_0,\eps$,
\begin{equation} \label{E:delta lambda smaller}
|\delta^{v_0 \cdot \deg I} \lambda_I(0)| < \delta^\eps|\delta^{v_0 \cdot b_0} \lambda_{I_0}(0)|, \qquad I \in (\scriptW_0 \cup \scriptW_1)^d, \qquad \deg I \neq b_0.
\end{equation}

By the Jacobi identity, if $w,w' \in \scriptW_0$,
\begin{equation} \label{E:jacobi2}
[\delta^{v_0 \cdot \deg w} X_w,\delta^{v_0 \cdot \deg w'} X_{w'}] = \sum_{\deg \tilde w = \deg w + \deg w'} C_{w,w'}^{\tilde w}(\delta^{v_0 \cdot \deg \tilde w} X_{\tilde w}),
\end{equation}
for admissible (because $\scriptW_0$ is) constants $C_{w,w'}^{\tilde w}$.  If $v_0 \cdot (\deg w + \deg w') \leq d$, each $\tilde w$ in the sum is an element of $\scriptW_0$.  If not, each $\tilde w$ is in $\scriptW_1$, and we can expand
$$
\delta^{v_0 \cdot \deg \tilde w} X_{\tilde w} = \sum_{j=1}^d \delta^{v_0 \cdot \deg \tilde w} X_{\tilde w}^j \partial_j = \sum_{j=1}^d (\delta^{v_0 \cdot \deg \tilde w - v_0^j} X_{\tilde w}^j)(\delta^{v_0^j}X_j).
$$
Note that $v_0 \cdot \deg \tilde w - v_0^j > 0$ for $\tilde w \in \scriptW_1$.  Using \eqref{E:jacobi2} to put the pieces back together, for sufficiently small $\delta > 0$ and any $w,w' \in \scriptW_0$, 
$$
[\delta^{v_0 \cdot \deg w} X_w,\delta^{v_0 \cdot \deg w'} X_{w'}] = \sum_{\tilde w \in \scriptW_0} c_{w,w'}^{\tilde w,\delta} \delta^{v_0 \cdot \deg \tilde w} X_{\tilde w},
$$
with
\begin{equation} \label{E:cwww smooth}
\|c_{w,w'}^{\tilde w,\delta}\|_{C^M(V)} \lesssim 1.
\end{equation}

The conclusion of the lemma is now a direct application of Theorem~5.3 of \cite{Street}, whose (lengthy) proof uses compactness arguments and Gromwall's inequality, among other tools.  For the convenience of the reader wishing to verify this, we provide a short dictionary to translate the notation.  Let $M$ be sufficiently large (depending on $m,d,I_0$) and choose $\delta(m)>0$ sufficiently small that \eqref{E:deltaXsmooth}, \eqref{E:delta lambda smaller}, and \eqref{E:cwww smooth} all hold.  Then the terms
$$
\{X_1,\ldots,X_q\}, \{d_1,\ldots,d_q\}, \scriptA, (\delta^d X), n_0(x,\delta)
$$
from \cite{Street} are, in our notation,
$$
\{X_w\}_{w \in \scriptW_0}, \{\deg w\}_{w \in \scriptW_0}, \{(\delta^{v_0^1}, \ldots,\delta^{v_0^k}) : 0 < \delta \leq \delta(m)\}, (\delta^{v_0 \cdot \deg w} X_w)_{w \in \scriptW_0}, d.
$$

\textit{A priori}, the results of \cite{Street} only guarantee that for each $m \geq 0$, there exists an admissible constant $\eta>0$ such that the conclusions hold on $B(\eta)$.  We want $\eta = 1$, but this is just a matter of rescaling.  Define 
$$
D^\eta_{v_0,I_0}(t_1,\ldots,t_d) := (\eta^{v_0 \cdot \deg w_1}t_1,\ldots,\eta^{v_0 \cdot \deg w_d}t_d); 
$$
then
$$
\Phi^{\eta\delta} = \Phi^\delta\circ D^\eta_{v_0,I_0}, \qquad Y_w^{\eta\delta} = (D^\eta_{v_0,I_0})^{-1} \eta^{v_0 \cdot \deg w} Y_w \circ D^\eta_{v_0,I_0}.
$$
Thus the lemma holds with a slightly smaller ($\eta$ times the original) value of $\delta(M)$.  
\end{proof}

\begin{lemma} \label{L:holds for Yw}
Let $m$ be a sufficiently large admissible integer, and let $Y_1,\ldots,Y_k$ be vector fields with the properties that 
\begin{gather}
\label{E:YjCm}
\|Y_j\|_{C^m(B(1))} \lesssim 1,\\
\label{E:Ydet=1}
|\det(Y_{w_1},\ldots,Y_{w_d})| \sim 1 \quad \text{on $B(1)$};
\end{gather}
here we recall that $(w_1,\ldots,w_d) = I_0$.  For $J \in \{1,\ldots,k\}^d$, define
$$
\tilde\Psi^J(t_1,\ldots,t_d) := e^{t_d Y_{J_d}} \circ \cdots \circ e^{t_1 Y_{J_1}}(0).
$$
Then
\begin{equation} \label{E:maxJdet}
\max_{J \in \{1,\ldots,k\}^d} \|\det D\tilde\Psi^J\|_{C^0(B(c_0))} \sim 1,
\end{equation}
for some admissible constant $c_0 >0$; in particular, $\tilde\Psi^J$ is defined on the ball $B(c_0)$.
\end{lemma}

\begin{proof}
There are similar results in \cite{ChRl, CNSW, BSrevista, TW}, but without the uniformity, so we give a complete proof.  

The upper bound, $\|\det D\tilde\Psi^J\|_{C^0(B(c_0))} \sim 1$ is an immediate consequence of \eqref{E:YjCm} for $m \geq 2$, by Picard's existence theorem.  

For the lower bound, we first show that if $m \geq |b_0|_1+2$, the left side of \eqref{E:maxJdet} is nonzero.  For $1 \leq i \leq d$ and $J \in \{1,\ldots,k\}^i$, define
$$
\tilde\Psi_i^J(t_1,\ldots,t_i) := e^{t_i Y_{J_i}} \circ \cdots \circ e^{t_1Y_{J_1}}(0);
$$
$\tilde\Psi_i^J \in C^{m+1}(B(c_0))$ for admissible $c_0>0$ by standard ODE existence results.  Supposing that the left side of \eqref{E:maxJdet} is zero, there exists some minimal $i \in \{0,\ldots,d-1\}$ such that
$$
\max_{J \in \{1,\ldots,k\}^{i+1}}\|\partial_{t_1}\tilde\Psi_{i+1}^J \wedge \cdots\wedge\partial_{t_{i+1}}\tilde\Psi_{i+1}^J\|_{C^0(B(c_0))} = 0.
$$
By \eqref{E:Ydet=1}, the $Y_j$ cannot all vanish at zero, so this $i$ is at least 1.  

By minimality of $i$, there exist $J \in \{1,\ldots,k\}^i$, $t_0 \in \R^i$ with $|t_0| < c_0$, and $\eps > 0$ such that $\tilde\Psi_i^J$ is an injective immersion on $\{t \in \R^i : |t-t_0| < \eps\} =: B_{t_0}(\eps)$.  Our assumption and the definition of exponentiation imply that for all $1 \leq j \leq k$ and $(t_1,\ldots,t_i) \in B(c_0)$,
\begin{align*}
0 &= (\partial_{t_1}\tilde\Psi_{i+1}^{(J,j)} \wedge \cdots\wedge \partial_{t_{i+1}} \tilde\Psi_{i+1}^{(J,j)})(t_1,\ldots,t_i,0)\\
&= (\partial_{t_1}\tilde\Psi_i^J \wedge \cdots \wedge \partial_{t_i}\tilde\Psi_i^J)(t_1,\ldots,t_i) \wedge Y_j(\tilde\Psi_i^J(t_1,\ldots,t_i)).
\end{align*}
Therefore $Y_1,\ldots,Y_k$ are tangent to $\tilde\Psi_i^J(B_{c_0}(\eps))$, as must be any Lie brackets that are defined, in particular all of those up to order $m$.  Since $m \geq |b_0|_1$, this contradicts \eqref{E:Ydet=1}.  Tracing back, we see that we must have $\det \tilde\Psi^J \not\equiv 0$ on $B(c_0)$ for some $J \in \{1,\ldots,k\}^d$.  

Now we prove that there is a uniform lower bound for $m:= |b_0|_1+3$.  If not, there exists a sequence $(Y_1^{(n)},\ldots,Y_k^{(n)})$ satisfying hypotheses \eqref{E:YjCm} and \eqref{E:Ydet=1}, but with
$$
\max_{J \in \{1,\ldots,k\}^d} \|\det D\tilde\Psi^{(n),J}\|_{C^0(B(c_0))} \to 0,
$$
where $\tilde \Psi^{(n),J}(t_1,\ldots,t_d) := e^{t_dY_{J_d}^{(n)}} \circ \cdots \circ e^{t_1Y_{J_1}^{(n)}}(0)$.  By Arzela--Ascoli, after passing to a subsequence, each $(Y_j^{(n)})$ converges in $C^{m-1}(B(1))$ to some vector field $Y_j$.  Thus for $|\deg w|_1 \leq m-1$, $Y_w^{(n)} \to Y_w$, and by standard ODE results, for each $J$, the sequence $(\tilde\Psi^{(n),J})$ converges to $\tilde\Psi^J$ in $C^m(B(c_0))$.  So $Y_1,\ldots,Y_k$ satisfy hypotheses \eqref{E:YjCm} and \eqref{E:Ydet=1} (the former with $m=|b_0|_1+2$), but $\det D\tilde\Psi^J \equiv 0$ on $B(c_0)$, for all $J \in \{1,\ldots,k\}^d$.  This is impossible, so the lower bound in \eqref{E:maxJdet} must hold.  
\end{proof}

We return to a consideration of the vector fields $X_1,\ldots,X_k$ in the next lemma, where we transfer the inequality in Lemma~\ref{L:holds for Yw} from $\tilde\Psi^J$ to $\Psi^J$.

\begin{lemma} \label{L:dadetDPsiJ}
For $J \in \{1,\ldots,k\}^d$ and $\alpha \in \Z_0^d$, if $v_0 \cdot(\deg J + \deg_J\alpha) < v_0 \cdot b_0$, then $\partial^\alpha \det D\Psi^J(0) = 0$.  Furthermore,
\begin{equation} \label{E:dadetDPsiJ}
\sum_{J \in \{1,\ldots,k\}^d} \sum_{\underset{v_0 \cdot (\deg J + \deg_J\alpha) = v_0 \cdot b_0}{\alpha \in (\Z_0)^d}}|\partial^\alpha\det D\Psi^J(0)| \sim |\lambda_{I_0}(0)|.
\end{equation}
\end{lemma}

\begin{proof}
For $J \in \{1,\ldots,k\}^d$, let
\begin{align*}
&\Psi^{J,\delta} := \Psi^J \circ D_J^\delta, \qtq{where} D_J^\delta(t_1,\ldots,t_d):= (\delta^{v_0^{J_1}}t_1,\ldots,\delta^{v_0^{J_d}}t_d),\\
&\tilde\Psi^{J,\delta} := e^{t_d Y_{J_d}^\delta} \circ \cdots \circ e^{t_1 Y_{J_1}^\delta}(0),
\end{align*}
with $Y_1^\delta,\ldots,Y_k^\delta$ as in \eqref{E:define Y delta}.  By naturality of exponentiation, $\Psi^{J,\delta} = \Phi^\delta \circ \tilde\Psi^{J,\delta}$, where $\Phi^\delta$ is defined in \eqref{E:define Phi delta}.  Hence by Lemmas~\ref{L:Ys} and~\ref{L:holds for Yw},
\begin{equation} \label{E:maxDPsiJ}
\max_{J \in \{1,\ldots,k\}^d} \|\det D\Psi^{J,\delta}\|_{C^0(B(c_0))} \sim \delta^{v_0 \cdot b_0} |\lambda_{I_0}(0)|, \qquad 0 < \delta < \delta(m), 
\end{equation}
where $m=m(b_0,d)$ is sufficiently large and $\delta(m)$ is the (inadmissible) constant from Lemma~\ref{L:Ys}.  As we will see, the lemma follows by sending $\delta\searrow0$.  

Let $M = M(b_0,d)$ be a sufficiently large integer, let $J \in \{1,\ldots,k\}^d$, and let $P^{J,\delta}$ be the degree $M$ Taylor polynomial of $\det D\Psi^{J,\delta}$, centered at 0.  Then
\begin{align}\notag
&\|P^{J,\delta}-\det D\Psi^{J,\delta}\|_{C^0(B(c_0))} = (\tfrac{\delta}{\delta(m)})^{v_0 \cdot \deg J} \|P^{J,\delta(m)}-\det D\Psi^{J,\delta(m)}\|_{C^0(D^{\delta/\delta(m)}B(c_0))}\\\notag
&\qquad\qquad \lesssim (\tfrac\delta{\delta(m)})^{v_0 \cdot \deg J+(M+1)\min_i v_0^i} \|\det D\Psi^{J,\delta(m)}\|_{C^0(D^{\delta/\delta(m)}B(c_0))}\\\label{E:PolyErr}
&\qquad\qquad\lesssim(\tfrac\delta{\delta(m)})^{v_0 \cdot \deg J + (M+1)\min_i v_0^i}, 
\end{align}
where the first inequality is by Taylor's theorem and admissibility of $M$, and the second is from \eqref{E:deltaXsmooth}, provided $m$ is sufficiently large depending on $M$.  Motivated by this inequality, we assume that $v_0 \cdot b_0 < M \min_i v_0^i$.

By the equivalence of all norms on the space of degree at most $M$ polynomials of $d$ variables,
\begin{equation} \label{E:PolyNorms}
\|P^{J,\delta}\|_{C^0(B(c_0))} \sim \sum_{|\alpha|_1 \leq M} |\partial^\alpha P^{J,\delta}(0)| = \sum_{|\alpha|_1 \leq M} \delta^{v_0 \cdot(\deg J + \deg_J\alpha)}|\partial^\alpha \det D\Psi^J(0)|.
\end{equation}
If $\alpha \in \Z_0^d$ and $v_0 \cdot(\deg J + \deg_J\alpha) \leq v_0 \cdot b_0$, then $|\alpha|_1 \leq \tfrac1{\min_iv_0^i}(v_0 \cdot \deg_J\alpha) \leq M$, and 
\begin{align*}
&\delta^{v_0 \cdot(\deg J + \deg_J\alpha)}|\partial^\alpha \det D\Psi^J(0)| = |\partial^\alpha P^{J,\delta}(0)| \lesssim \|P^{J,\delta}\|_{C^0(B(c_0))}\\
&\qquad\qquad\lesssim\|\det D\Psi^{J,\delta}\|_{C^0(B(c_0))} + (\tfrac\delta{\delta(m)})^{v_0 \cdot \deg J + (M+1)\min_i v_0^i}\\
&\qquad\qquad \lesssim \delta^{v_0 \cdot b_0}|\lambda_{I_0}(0)| + (\tfrac\delta{\delta(m)})^{v_0 \cdot \deg J + (M+1)\min_i v_0^i}.
\end{align*}
Sending $\delta \searrow 0$, we see that
\begin{align}
\label{E:detDPsiJ0}
&\partial^\alpha\det D\Psi^J(0) = 0, \qtq{whenever} v_0 \cdot(\deg J + \deg_J\alpha) < v_0 \cdot b_0, \\
\label{E:detDPsiJsmall}
&|\partial^\alpha\det D\Psi^J(0)| \lesssim |\lambda_{I_0}(0)| \qtq{if} v_0 \cdot(\deg J + \deg_J\alpha) = v_0 \cdot b_0.
\end{align}

Now for the lower bound.  By \eqref{E:maxDPsiJ} and the fact that there are only finitely many choices for $J$, there exist $J \in \{1,\ldots,k\}^d$ and a sequence $\delta_n \searrow 0$ such that
\begin{equation} \label{E:detDPsiJbig}
\|\det D\Psi^{J,\delta_n}\|_{C^0(B(c_0))} \gtrsim \delta_n^{v_0 \cdot b_0} |\lambda_{I_0}(0)|.
\end{equation}
Since $M \min_i v_0^i > v_0 \cdot b_0$ and $\lambda_{I_0}(0) \neq 0$, \eqref{E:detDPsiJbig}, \eqref{E:PolyErr}, and \eqref{E:PolyNorms} imply that for $\delta_n$ sufficiently  (inadmissibly) small,
$$
\delta_n^{v_0 \cdot b_0} |\lambda_{I_0}(0)| \lesssim \|P^{J,\delta_n}\|_{C^0(B(c_0))} \lesssim \sum_{|\alpha|_1 \leq M} \delta_n^{v_0 \cdot(\deg J + \deg_J\alpha)}|\partial^\alpha \det D\Psi^J(0)|.
$$
Applying \eqref{E:detDPsiJ0} and letting $n \to \infty$, 
$$
|\lambda_{I_0}(0)| \lesssim \sum_{v_0 \cdot(\deg J + \deg_J\alpha) = v_0 \cdot b_0} |\partial^\alpha \det D\Psi^J(0)|.
$$
This completes the proof of \eqref{E:dadetDPsiJ}, and thus of Lemma~\ref{L:dadetDPsiJ}.
\end{proof}

By our choice of $v_0$, \eqref{E:dadetDPsiJ} is just \eqref{E:main equivalence}, so to complete the proof of Proposition~\ref{P:equivalence}, it suffices to prove the following.

\begin{lemma} \label{L:one P0}
$\scriptP_0 = \widetilde\scriptP_0$.
\end{lemma}

\begin{proof}
By \eqref{E:main equivalence}, $\tilde\scriptP_0$ contains the extreme points of $\scriptP_0$, so $\scriptP_0 \subseteq \tilde\scriptP_0$.  Now suppose that $b_0 \notin \scriptP_0$.  Then there exist $v_0 \in (0,1]^k$ and $\eps > 0$ such that $v_0 \cdot b_0 + \eps < v_0 \cdot p$, for all $p \in \scriptP_0$.  At least one extreme point $b$ of $\scriptP_0$ satisfies $v_0 \cdot b = \max_{p \in \scriptP_0} v_0 \cdot p$; perturbing $v_0$ slightly, we may assume that there exists $b_1 \in \scriptP_0$ such that
$$
v_0 \cdot b_0 < v_0 \cdot b_1 < v_0 \cdot p, \qquad \qtq{for all} b_1 \neq p \in \scriptP_0.
$$
By Lemma~\ref{L:dadetDPsiJ}, $\partial^\alpha \det D\Psi^J(0) = 0$ whenever $(\deg J + \deg_J\alpha) \cdot v_0 < v_0 \cdot b_1$, so $b_0 \notin \tilde\scriptP_0$.  Thus $\scriptP_0 \subseteq \tilde\scriptP_0$, and we are done.
\end{proof}

\subsection*{Remarks}  A more direct argument, using the Baker--Campbell--Hausdorff formula should be possible, but the author has not been able to carry this out.  Let $k=d$ and consider vector fields $X_1,\ldots, X_d$.  Using the approximation $\exp(tX) = \sum_{n=0}^N \tfrac{t^n}{n!} X^{n-1}(X) + O(|t|^N)$, which may be found in \cite{CNSW}, the formula for the Lie derivative of a determinant of $d$ vector fields, and somewhat tedious computations, one can show that
\begin{align*}
&\partial_t^\alpha|_{t=0} \det D_t \bigl(e^{t_dX_d} \circ \cdots \circ e^{t_1 X_1}\bigr)(x_0) \\ &\qquad = \pm \sum_{w_1,\ldots,w_d}^* \prod_{i=1}^d \binom{\alpha_i}{\deg_i w_{i+1},\ldots, \deg_i w_d} \det(X_{w_1},X_{w_2},\ldots,X_{w_d}),
\end{align*}
where the $*$ indicates that the sum is taken over those words $w_i = (w_i^1,\ldots,w_i^{n_i})$ satisfying $\sum_i \deg w_i = \alpha + (1,\ldots,1)$ and $w_i^1 = i > w_i^2 \geq \cdots \geq w_i^{n_i}$ (in particular, $w_1=(1)$).  Replacing $X_i$ above with $X_{J_i}$ gives an alternative proof that the right (Jacobian) side of \eqref{E:main equivalence} is bounded by the left (determinant) side, but using this formula to bound the left of \eqref{E:main equivalence} by the right seems nontrivial.

The estimate \eqref{E:main equivalence} may fail if $b$ is not extreme (even if it is minimal).  To see this, let $\gamma(t) := (t,\ldots,t^d)$ and define $X_0 := \partial_t$, $X_i := \partial_t - \gamma'(t)\cdot \nabla_x$, $1 \leq i \leq d$, and take $b := (1+\tfrac{d(d-1)}2,1,\ldots,1)$.  In this case, the only $I$ with $\deg I=b$ and $\lambda_I \not\equiv 0$ are those of the form 
$$
I=((1),(j_1),(1,j_2),\ldots,(1,\ldots,1,j_d)),
$$
with the $j_i$ distinct.  Thus the left side of \eqref{E:main equivalence} is a non-zero dimensional constant.  On the other hand, simple combinatorial considerations show that the right side of \eqref{E:main equivalence} must be identically zero.  

Less uniform versions of \eqref{E:main equivalence} may be found in \cite{CNSW, BSrevista, TW}.  Let $X_1,\ldots,X_k$ be smooth vector fields and assume that there exists a $d$-tuple $I = (w_1,\ldots,w_d)$ such that $|\lambda_I| \geq 1$ on $U$.  Let  $\delta_1,\ldots,\delta_k$ be scalars satisfying the smallness and weak comparability conditions
$$
\delta_i \leq K, \qquad \delta_i \leq K\delta_j^\eps, \qquad 1 \leq i,j \leq k.
$$
Then \cite{TW, BSrevista} prove that there exist $N \geq |\deg I|_1$ and $N'$ (depending on $I$) such that
\begin{align*}
&\sum_{|\deg I|_1 \leq N} \bigl(\prod_{i=1}^k \delta_i^{(\deg I)_i}\bigr) |\lambda_I(x_0)| 
\\ & \qquad \sim \sum_{J \in \{1,\ldots,k\}^d} \sum_{\underset{\deg J + \deg_J \alpha \leq N'}{\alpha \in (\Z_0)^d}} \bigl(\prod_{i=1}^k \delta_i^{\deg J + \deg_J \alpha}\bigr)|\partial^\alpha_t \det D_t\Psi^J_{x_0}(0)|, \qquad x_0 \in U,
\end{align*}
with inadmissible implicit constants.  It is not shown, however, how to remove the dependence of the implicit constant on $\eps$, $K$, or the $X_i$, or, in particular, how to remove the assumption that the H\"ormander condition holds uniformly.

%%%%%%%%%%%%%%%%%%%%%%%%%%%%%%%%%%%%%%
%%%%%%%%%%%%%%%%%%%%%%%%%%%%%%%%%%%%%%
%%%%%%%%%%%%%%%%%%%%%%%%%%%%%%%%%%%%%%

\section{Proof the optimality result:  Proposition~\ref{P:optimality}} \label{S:optimality}

%%%%%%%%%%%%%%%%%%%%%%%%%%%%%%%%%%%%%%
%%%%%%%%%%%%%%%%%%%%%%%%%%%%%%%%%%%%%%
%%%%%%%%%%%%%%%%%%%%%%%%%%%%%%%%%%%%%%

The entirety of this section will be devoted to the proof of Proposition~\ref{P:optimality}.  It suffices to prove the proposition when $\supp \mu \subseteq V$, and $V$ and $W$ are bounded open subsets of $U$ with $\overline V \subseteq W$, $\overline W \subseteq U$.  (Recall that $U$ is the set on which the $\pi_i$, and hence the $X_i$, are defined.)  By \eqref{E:M with mu} with $E_i = \pi_i(V)$, $1 \leq i \leq k$, $\mu(V) < \infty$.  

Throughout this section, an object will be said to be admissible if it depends (or it is taken from a finite set depending) only on $d$ and $p = (p_1,\ldots,p_k)$.  All implicit constants will be admissible.  The constant $A(\mu)$ will always represent precisely the quantity in \eqref{E:M with mu}, and in particular will not be allowed to change from line to line.  

First suppose that $p_{j_0} < 1$.  Without loss of generality, $j_0=1$.  We may cover $\pi_1(V)$ by $C_{V,\pi_1} \eps^{-(d-1)}$ balls $B_i$ of radius $\eps$, so
\begin{align*}
\mu(V) &\leq \sum_i \int \chi_{B_1}\circ \pi_1 \prod_{j=2}^k \chi_{\pi_j(V)}\circ \pi_j \, d\mu \leq A(\mu) \sum_i |B_1|^{1/p_1} \prod_{j=2}^k |\pi_j(V)|^{1/p_j} \\
&\leq C(\mu, d, p, V, \pi_2,\ldots,\pi_k) \eps^{(d-1)(\frac1{p_1}-1)}.
\end{align*}
Letting $\eps \to 0$, we see that $\mu \equiv 0$.   

We now turn to the case when $\sum_j p_j^{-1}>1$.  Replacing $\{X_1,\ldots,X_k\}$ with $\{\partial_1,\ldots,\partial_d,X_1,\ldots,X_k\}$, $(p_1,\ldots,p_k)$ with $(\infty,\ldots,\infty,p_1,\ldots,p_k)$, and $k$ with $d+k$ if necessary, we may assume that $X_i = \partial_i$, $1 \leq i \leq d$, without affecting either of the following sets
\begin{align*}
Z &:= \{x \in V : b_p \notin \scriptP_x\}\\
\Omega &:= \{x \in V : b_p \ctc{is an extreme point of} \scriptP_x\},
\end{align*}
or the quantity on the right of \eqref{E:mu lesssim rho}. 

The proposition will follow from the next two lemmas.

\begin{lemma} \label{L:mu(Z)=0}
$\mu(Z) = 0$.
\end{lemma}

\begin{lemma} \label{L:mu(O')}
If $\rho := \sum_{\deg I = b_p} |\lambda_I|^{\frac1{|b_p|_1-1}}$ and 
$$
\Omega_n := \{x \in \Omega : 2^n \leq \rho(x) \leq 2^{n+1}\}, \qquad n \in \Z,
$$
then $\mu(\Omega') \lesssim A(\mu) 2^n |\Omega'|$ for any Borel set $\Omega' \subseteq \Omega_n$.
\end{lemma}

\begin{proof}[Proof of Lemma~\ref{L:mu(Z)=0}]
By Proposition~\ref{P:polytope prop}, there exist admissible, finite sets $\scriptA_i$, $i=1,\ldots,C_{p,d}$ such that $b_p \notin \scriptP(\scriptA_i)$ for any $i$ and for each $x \in Z$, there exists an $i$ such that $\scriptP_x \subseteq \scriptP(\scriptA_i)$.  For the remainder of the proof of the lemma, we let $\scriptA = \scriptA_i$ be fixed and define
$$
Z'  := \{x \in Z : \scriptP_x \subseteq \scriptP(\scriptA)\}.
$$
It suffices to show that $\mu(Z') = 0$.

Choose admissible $\eps>0$ and $v \in (\eps,1]^k$ such that
$$
v \cdot b_p + \eps < v \cdot b, \qtq{for} b \in \scriptP(\scriptA).
$$
Define
$$
\scriptW_0 := \{w \in \scriptW : v \cdot \deg w \leq d\}.
$$

Let $N=N_{d,p}$ be an integer whose size will be determined in a moment and which is, in particular, larger than $\frac{d}\eps$.  Since $\overline W$ is compact and contained in $U$, the $X_i$ are smooth on $U$, and $\{X_w : w \in \scriptW_0\}$ contains the coordinate vector fields, there exists $\delta_0 > 0$, depending on the $\pi_i$, $p$, and $W$, such that for all $0 < \delta \leq \delta_0$, $I \in \scriptW_0^d$ satisfying $\deg I \in \scriptP(\scriptA)$, $x \in W$, and $w, w' \in \scriptW_0$,
\begin{gather}\label{E:lambda I small}
|\delta^{v \cdot \deg I} \lambda_I(x)| < \delta^\eps \delta^{v \cdot b_p},\\ 
\label{E:deltaXw is small/smooth}
\|\delta^{v \cdot \deg w} X_w \|_{C^0(W)} \leq \tfrac1d \dist(V,\partial W), \qquad \|\delta^{v \cdot \deg w} X_w \|_{C^N(W)} \leq 1, \\ \notag
[\delta^{v \cdot \deg w} X_w, \delta^{v \cdot \deg w'} X_{w'}] = \sum_{\tilde w \in \scriptW_0} c_{w,w'}^{\tilde w,\delta} \delta^{v \cdot \deg \tilde w} X_{\tilde w}, 
\end{gather}
with
$$
\|c_{w,w'}^{\tilde w,\delta}\|_{C^N(W)} \lesssim 1.
$$
We omit the details since they are essentially the same as arguments found in the proof of Lemma~\ref{L:Ys}.

For $x \in Z'$ and $0 < \delta \leq \delta_0$, choose $I_x^\delta \in \scriptW_0^d$ such that
$$
\delta^{v \cdot \deg I_x^\delta} |\lambda_{I_x^\delta}(x)| = \max_{I \in \scriptW_0^d} \delta^{v \cdot \deg I} |\lambda_I(x)|.
$$
Let 
\begin{equation} \label{E:def B}
\begin{gathered}
\Phi_x^\delta(t_1,\ldots,t_d) := \exp(t_1 \delta^{v \cdot \deg w_1} X_{w_1} + \cdots + t_d \delta^{v \cdot \deg w_d} X_{w_d})(x)\\
B(x,\delta) := \{\Phi_x^\delta(t) : |t| < 1\},
\end{gathered}
\end{equation}
where $I_x^\delta = (w_1,\ldots,w_d)$.  Then $B(x,\delta) \subseteq W$ by \eqref{E:deltaXw is small/smooth} and the fact that $x \in Z' \subseteq V$.

By the results of \cite{Street}, provided $N=N_{d,p}$ is sufficiently large, these balls are doubling in the sense that $|B(x,\delta)| \sim |B(x,2\delta)|$, for all $x \in Z'$ and $0 < \delta \leq \delta_0$.  (Here we are using the fact that $\eps$ and $v$ are admissible.)  Furthermore, for $x \in V$,
\begin{gather} \label{E:size Bz}
|B(x,\delta)| \sim \delta^{v \cdot \deg I_x^\delta}|\lambda_{I_x^\delta}(x)|\\
\label{E:size alphai}
\exp(tX_i)(y) \in B(x,C\delta)\qtq{whenever} y \in B(x,\delta), \quad |t| < \delta^{v^i},
\end{gather}
where $C=C_{d,p}$.  By the doubling property, the change of variables formula, and \eqref{E:size alphai}, if $\sigma_i:\pi_i(W) \to \R^d$ is any smooth section of $\pi_i$ (i.e.\ $\sigma_i \circ \pi_i$ is the identity), with $\sigma_i(\pi_i(V)) \subseteq W$, 
\begin{equation} \label{E:compare size}
\begin{aligned}
|B(x,\delta)| &\sim |B(x,C\delta)| = \int_{\pi_i(B(x,C\delta))} \int_\R \chi_{B(x,C\delta)}(e^{tX_i}(\sigma_i(y))\, dt\, dy\\
& \geq  \int_{\pi_i(B(x,\delta/2))} \int_\R \chi_{B(x,C\delta)}(e^{tX_i}(\sigma_i(y)))\, dt \, dy \gtrsim  \delta^{v^i} |\pi_i(B(x,\delta))|.
\end{aligned}
\end{equation}

By the Vitali covering lemma (as stated in \cite{BigStein}, for instance), for each $0 < \delta  \leq \delta_0$, there exists a collection of points $\{x_j\}_{j=1}^{M_\delta} \subseteq Z'$ such that $Z' \subseteq \bigcup_{j=1}^{M_\delta} B(x_j,\delta)$ and such that the balls $B(x_j,C^{-1}\delta)$ are pairwise disjoint.  
By this, \eqref{E:M with mu} and the fact that $\chi_{B(x_j,\delta)} \leq \prod_{i=1}^k \chi_{\pi_i(B(x_j,\delta))} \circ \pi_i$, \eqref{E:compare size}, \eqref{E:size Bz} and the definition of $b_p$, the doubling property and \eqref{E:lambda I small}, and finally, disjointness of the $B(x_j,\delta)$, 
\begin{align*}
\mu(Z') &\leq \sum_{j=1}^{M_\delta} \mu(B(x_j,\delta)) \leq A(\mu) \sum_j \prod_{i=1}^k |\pi_i(B(x_j,\delta))|^{\frac1{p_i}} \\
&\lesssim A(\mu)  \sum_j |B(x_j,C\delta)|^{\sum_i \frac1{p_i}} \prod_i \delta^{-\frac{v^i}{p_i}}\\
&\sim A(\mu) \sum_j |B(x_j,C\delta)| (\delta^{v \cdot \deg I_{x_j}^\delta - v \cdot b_p}|\lambda_{I_{x_j}^\delta}(x_j)|)^{\sum_i \frac1{p_i}-1 }\\
&\lesssim A(\mu)\sum_j |B(x_j,C^{-1}\delta)|\delta^{\eps(\sum_i \frac1{p_i}-1)} \leq A(\mu) |W| \delta^{\eps(\sum_i \frac1{p_i}-1)}.
\end{align*}
The lemma follows by sending $\delta$ to 0.  
\end{proof}

\begin{proof}[Proof of Lemma~\ref{L:mu(O')}] The proof is similar to that of Lemma~\ref{L:mu(Z)=0}.  Fix $n$ and $\Omega' \subseteq \Omega_n$.  Let $x \in \Omega'$.  Since $\Omega' \subseteq \Omega$, $b_p$ is an extreme point of $\scriptP_x$.  By the definition of $\rho$, $\max_{\deg I = b_p} |\lambda_I(x)| \sim 2^{n(|b_p|_1-1)}$.  

By Proposition~\ref{P:polytope prop} and a covering argument, we may assume that there exists a finite set $\scriptA \subseteq \Z_0^k$ such that $b_p \notin \scriptP(\scriptA)$ and for each $x \in \Omega'$, $\scriptP_x \subseteq \scriptP(\scriptA \cup \{b_p\})$.  Choose $\eps > 0$, $v \in (\eps,1]^k$ such that $v \cdot b_p + \eps < v \cdot b$ for each $b \in \scriptP(\scriptA \cup \{b_p\})\cap \Z_0^k \setminus \{b_p\}$, and let 
$$
\scriptW_0 := \{w \in \scriptW : v \cdot \deg w \leq d\}.
$$
Since $(1,\ldots,1,0,\ldots,0) \in \scriptP_x$ for each $x \in U$, $(1,\ldots,1,0,\ldots,0) \in \scriptP(\scriptA\cup\{b_p\})$.  Therefore $v \cdot b_p \leq \sum_{i=1}^d v^i \leq d$, so $\deg I = b_p$ implies that $I \in \scriptW_0^d$.  

Let $N = N_{d,p}$ be a large integer.  As before, there exists $\delta_n > 0$, which depends on $n$, the $\pi_i$, and on $p$, such that for all $0 < \delta \leq \delta_n$, $x \in \Omega'$, $I \in \scriptW_0^d$ with $\deg I \neq b_p$, and $w,w' \in \scriptW_0$,  
\begin{gather*}
|\delta^{v \cdot \deg I} \lambda_I(x)| < \delta^\eps \max_{\deg I' = b_p} \delta^{v \cdot \deg I'}|\lambda_{I'}(x)|, \\
\|\delta^{v \cdot \deg w}X_w\|_{C^0(W)} \leq \tfrac1d \dist(V,\partial W), \qquad \|\delta^{v \cdot \deg w} X_w\|_{C^N(W)} \leq 1,\\
[\delta^{v \cdot \deg w}X_w, \delta^{v \cdot \deg w'} X_{w'}] = \sum_{\tilde w \in \scriptW_0} c_{w,w'}^{\tilde w,\delta} \delta^{v \cdot \deg \tilde w} X_{\tilde w},
\end{gather*}
with 
$$
\|c_{w,w'}^{\tilde w,\delta} \|_{C^N(W)} \leq C_{d,p},
$$
for all $w, w' \in \scriptW_0$.  In particular, we may choose $\delta_n$ sufficiently small that for each $x \in \Omega'$ and $0 < \delta \leq \delta_n$, there exists a $d$-tuple $I_x^\delta \in \scriptW_0^d$ such that $\deg I_x^\delta = b_p$ and
$$
\delta^{v \cdot \deg I_x^\delta} |\lambda_{I_x^\delta}(x)| = \max_{I \in \scriptW_0^d} \delta^{v \cdot \deg I} |\lambda_I(x)| \sim \delta^{v \cdot b_p} 2^{n(|b_p|_1-1)}.
$$
Thus, considering the balls $B(x,\delta)$ (defined in \eqref{E:def B}) for $x \in \Omega'$ and $0 < \delta \leq \delta_n$,
$$
|B(x,\delta)| \sim 2^{n(|b_p|_1-1)}\delta^{v \cdot b_p} = 2^{\frac n{\sum_i \frac1{p_i}-1}} \delta^{v \cdot b_p}.
$$

Since the balls $B(x,\delta)$ are doubling, for each $\eta > 0$ there exist a collection $\{x_j\}_{j=1}^{M_\delta} \subseteq \Omega'$ and a parameter $0 < \delta \leq \delta_n$ such that
$$
\Omega' \subseteq \bigcup_{j=1}^{M_\delta}B(x_j,\delta), \qquad |\bigcup_{j=1}^{M_\delta} B(x_j,\delta)| \leq |\Omega'|+\eta,
$$
and such that the $B(x_j,C^{-1}\delta)$ are pairwise disjoint.  

Arguing as in the proof of Lemma~\ref{L:mu(Z)=0},
\begin{align*}
\mu(\Omega') &\leq \sum_{j=1}^{M_\delta} \mu(B(x_j,\delta))
 \lesssim A(\mu) \sum_j |B(x_j,\delta)| |B(x_j,\delta)|^{\sum_i \frac1{p_i}-1} \delta^{-v \cdot b_p(\sum_i \frac1{p_i}-1)} \\
&\sim A(\mu)\sum_j |B(x_j,\delta)| 2^n \lesssim A(\mu) 2^n(|\Omega'| + \eta).
\end{align*}
Letting $\eta \to 0$ completes the proof.
\end{proof}

\subsection*{Remarks} The pointwise upper bound \eqref{E:mu lesssim rho} is false if no assumptions are made on $b_p$.  Indeed, if $b_p$ lies in the interior of $\scriptP_{x_0}$, then for some $\theta < 1$, $b_{\theta p}$ lies in the interior of $\scriptP_{x_0}$, where $\theta p = (\theta p_1,\ldots,\theta p_k)$.  Thus for some neighborhood $U$ of $x_0$, $b_{\theta p}$ lies in the interior of $\scriptP_x$ for every $x \in U$.  Hence by the main result in \cite{BSrevista}, if $a$ is continuous with compact support in $U$, 
$$
|\int \prod_{j=1}^k f_j \circ \pi_j (x) \, a(x)\, dx| \lesssim \prod_{j=1}^k \|f_j\|_{L^{\theta p_j}}.
$$
Additionally, 
$$
|\int \prod_{j=1}^k f_j \circ \pi_j(x) \, |\log|x-x_0|| \, a(x)\, dx| \lesssim \prod_{j=1}^k \|f_j\|_{L^{\infty}}.
$$
Thus by interpolation, 
$$
|\int \prod_{j=1}^k f_j \circ \pi_j(x) \, |\log|x-x_0||^{1-\theta} \, a(x)\, dx| \lesssim \prod_{j=1}^k \|f_j\|_{L^{p_j}}.
$$

For the unweighted bilinear operator in the `polynomial-like' case, the endpoint restricted weak type bounds are known and are due to Gressman in \cite{GressIMRN}; in the multilinear case, the corresponding estimates follow by combining his techniques with arguments in \cite{BSrevista}.  The deduction of endpoint bounds from the arguments in \cite{GressIMRN} does not seem to be immediate in the weighted case, and so these questions remain open except for certain special configurations (such as convolution or restricted X-ray transform along polynomial curves).

%%%%%%%%%%%%%%%%%%%%%%%%%%%%%%%%%%%%%%
%%%%%%%%%%%%%%%%%%%%%%%%%%%%%%%%%%%%%%
%%%%%%%%%%%%%%%%%%%%%%%%%%%%%%%%%%%%%%

\section{Proof of the main theorem:  Theorem~\ref{T:main}} \label{S:main}

%%%%%%%%%%%%%%%%%%%%%%%%%%%%%%%%%%%%%%
%%%%%%%%%%%%%%%%%%%%%%%%%%%%%%%%%%%%%%
%%%%%%%%%%%%%%%%%%%%%%%%%%%%%%%%%%%%%%

In this section, undecorated constants and implicit constants ($C,c,\lesssim, \gtrsim,\sim$) will be allowed to depend on a cutoff function $a$ (specifically, on upper bounds for $\diam(\supp a)$ and $\|a\|_{L^\infty}$), a point $b_0 \in \Z_0^k$, and exponents $p_1,\ldots,p_k$ (all of which will be given in a moment), as well as the $\pi_j$.  Other parameters (namely, $\eps,\delta,N$) that depend on $b_0,p_1,\ldots,p_k$ will arise later on, so implicit constants may depend on these quantities as well.  Unless otherwise stated, decorated constants and implicit constants ($c_d$, $\lesssim_{N,d}$, etc.) will only be allowed to depend on the objects in their subscripts.  

Let $J_0 \in \{1,\ldots,k\}^d$ and for $x \in U$, define $\Psi^{J_0}_x(t)$ as in \eqref{E:def Psi}.  Let $\beta_0$ be a multiindex, and define $b_0 := \deg J_0 + \deg_{J_0} \beta_0$.  Let
\begin{equation} \label{E:def tilde rho}
\tilde \rho(x) := |\partial_t^{\beta_0}|_{t = 0} \det D_t\Psi_x^{J_0}(t)|^{\frac1{|b_0|_1-1}}.
\end{equation}
Let $a$ be continuous and compactly supported in $U$, and define the multilinear form
$$
\widetilde M(f_1,\ldots,f_k) := \int_{\R^d} \prod_{j=1}^k f_j \circ \pi_j (x) \, \tilde\rho(x) \,a(x) \, dx.
$$

In light of Proposition~\ref{P:equivalence}, the following more general (we need not assume that $b_0$ is extreme) result implies Theorem~\ref{T:main}.  

\begin{theorem} \label{T:not extreme}  
Let $(p_1,\ldots,p_k) \in [1,\infty)^k$ satisfy $(p_1^{-1},\ldots,p_k^{-1}) \prec \mathbf q(b_0)$, with $p_i^{-1} < \mathbf q_i(b_0)$ when $b_0^i \neq 0$.  Then
\begin{equation} \label{E:not extreme}
|\widetilde M(f_1,\ldots,f_k)|  \lesssim \prod_{j=1}^k \|f_j\|_{L^{p_j}},
\end{equation}
for all continuous $f_1,\ldots,f_k$.
\end{theorem}

Since $J_0$ and $\beta_0$ are fixed, we will henceforth drop the tildes from our notation, with the understanding that we are using \eqref{E:def tilde rho} instead of \eqref{E:def rho} to define $\rho$.  

It suffices to prove \eqref{E:not extreme} when the $f_j$ are nonnegative.  Suppose that $b_j = 0$ for some $j$.  Then $\pi_j$ plays no role in the definition of $\rho$, and $p_j = \infty$, so by H\"older's inequality, we may ignore $f_j$ entirely.  Thus we may assume that $b_j \neq 0$ for each $j$.  In fact, we may assume that for each $j$, $p_j < \infty$ since $\|f_j\|_{L^{p_j}(\pi_j(\supp a))} \lesssim \|f_j\|_{L^\infty}$, by the compact support of $a$.  

We only claim a non-endpoint result, so by real interpolation with the trivial (by H\"older) inequalities of the form
$$
M(f_1,\ldots,f_k) \lesssim \prod_{j=1}^k \|f_j\|_{L^{\tilde p_j}}, \qquad \sum_{j=1}^k p_j^{-1} \leq 1,
$$
it suffices to prove that for all Borel sets $E_1,\ldots,E_k$ and some sufficiently small $\eps > 0$,
\begin{equation} \label{E:rwt 1}
\int_{\R^d} \prod_{j=1}^k \chi_{E_j} \circ \pi_j(x)\rho(x)a(x)\, dx \lesssim \prod_{j=1}^k |E_j|^{\mathbf q_j(b_0)-\eps}.
\end{equation}

Letting $\Omega := \supp a \cap \bigcap_{j=1}^k \pi_j^{-1}(E_j)$, \eqref{E:rwt 1} will follow from
\begin{equation} \label{E:rwt 2}
\rho(\Omega) \lesssim \prod_{j=1}^k |\pi_j(\Omega)|^{\mathbf q_j(b_0)-\eps}.
\end{equation}
If we define
\begin{equation} \label{E:define alphas}
\alpha_j := \frac{\rho(\Omega)}{|\pi_j(\Omega)|},
\end{equation}
a bit of arithmetic shows that \eqref{E:rwt 2} is equivalent to
$$
\prod_{j=1}^k \alpha_j^{\mathbf q_j(\mathbf q(b_0) - (\eps,\ldots,\eps))} \lesssim \rho(\Omega),
$$
which in turn would be implied by 
\begin{equation} \label{E:alpha lb 1}
\prod_{j=1}^k \alpha_j^{b_0^j+\eps} \lesssim \rho(\Omega),
\end{equation}
with a slightly smaller $\eps$.  (We recall that $\mathbf q$ equals its own inverse.)    

By the coarea formula,
\begin{equation} \label{E:coarea}
\alpha_j = |\pi_j(\Omega)|^{-1} \int_{\pi_j(\Omega)} \int_{\pi_j^{-1}\{y\}} \chi_\Omega(x) \rho(x) \tfrac{1}{|X_j(x)|} \, d\mathcal{H}^1(x)\, dy.
\end{equation}
Since $\pi_j$ is a submersion, $|X_j| \gtrsim 1$ and $\mathcal H^1(\pi_j^{-1}\{y\}) \lesssim 1$ for all $y \in \pi_j(\Omega)$.  Since $\rho \lesssim 1$ by smoothness of the $\pi_j$, \eqref{E:coarea} implies that
\begin{equation} \label{E:alpha lesssim diam}
\alpha_j \lesssim \diam(\Omega) \leq \diam(\supp a).
\end{equation}
By taking a partition of unity, we may assume that the $\alpha_j$ are as small as we like, in particular, that they are smaller than $\tfrac12$.  Reordering if necessary, $\alpha_1 \leq \cdots \leq \alpha_k$.  

For $n \in \Z$, let $\Omega_n = \{x \in \Omega : 2^n \leq \rho(x) < 2^{n+1}\}$.  Then for $C$ sufficiently large, $\Omega_n = \emptyset$ for all $n > C$.  On the other hand, since $\pi_1$ is a submersion and $\supp a$ is compact,
$$
\sum_{n \leq \log \alpha_1 - C} \rho(\Omega_n) \lesssim \sum_{n \leq \log \alpha_1 - C} 2^n |\pi_1(\Omega)| \lesssim 2^{-C}\alpha_1 |\pi_1(\Omega)| = 2^{-C}\rho(\Omega).
$$
Thus for $C$ sufficiently large, 
$$
\rho(\bigcup_{n \leq \log \alpha_1 - C} \Omega_n) < \tfrac12 \alpha_1 |\pi_1(\Omega)| = \tfrac12 \rho(\Omega).
$$
By pigeonholing, there exists $n$ with $\log \alpha_1 - C \leq n \leq C$ such that 
\begin{equation} \label{E:lb rho On}
\rho(\Omega_n) \geq (2(|\log\alpha_1|+2C))^{-1} \rho(\Omega) \gtrsim \alpha_1^\eps \rho(\Omega).
\end{equation}
Define
$$
\alpha_{n,j} := \tfrac{\rho(\Omega_n)}{|\pi_j(\Omega_n)|}, \qquad j=1,\ldots,k.
$$
By \eqref{E:lb rho On} and the triviality $\rho(\Omega_n) \leq \rho(\Omega)$, together with the proof of \eqref{E:alpha lesssim diam} and the small diameter of $\supp a$,
$$
\alpha_1^\eps \alpha_j \lesssim \alpha_{n,j} \leq \tfrac12.
$$
Therefore \eqref{E:alpha lb 1} follows from
\begin{equation} \label{E:alpha lb 2}
\rho(\Omega_n) \gtrsim \prod_{j=1}^k (\alpha_{n,j})^{b_0^j+\eps},
\end{equation}
with a slightly smaller value of $\eps$.  Henceforth, we let $\rho_0:= 2^n$ (for this value of $n$) and drop the $n$'s from the notation in \eqref{E:alpha lb 2}.  We note that $\rho(\Omega) \sim \rho_0 |\Omega|$.  Reordering again, we may continue to assume that $\alpha_1 \leq \ldots \leq \alpha_k$.  

Let $\delta > 0$ be a small constant (depending on $\eps, b_0, d$), which will be determined later on.  Cover $\Omega$ by $c_d \alpha_1^{-\delta d}$ balls of radius $\alpha_1^\delta$.  By pigeonholing, there exists $\Omega' \subseteq \Omega$ with 
$$
\rho(\Omega') \gtrsim \alpha_1^{\delta d} \rho(\Omega).
$$
Arguing as above, the parameters $\alpha_j':= |\pi_j(\Omega')|^{-1}\rho(\Omega')$ satisfy
\begin{equation} \label{E:alphaj'}
 \alpha_1^{1+\delta d} \leq \alpha_1^{\delta d}\alpha_j \lesssim \alpha_j' \lesssim \diam(\Omega') \leq \alpha_1^\delta.
\end{equation}
Thus for $\delta$ sufficiently small, \eqref{E:alpha lb 2} would follow from
$$
\rho(\Omega') \gtrsim \prod_{j=1}^k (\alpha_j')^{b_0^j+\eps},
$$
with a slightly smaller value of $\eps$.  

Since $\alpha_j' \lesssim \diam (\supp a)$, we may assume that the $\alpha_j'$ are as small (depending on the $\pi_j$, $\eps$, $\delta$),  as we like.  Thus \eqref{E:alphaj'} implies that for each $1 \leq j \leq k$,  
$$
\diam(\Omega') \leq c (\alpha_j')^\delta,
$$
for some slightly smaller value of $\delta$, and with $c$ as small as we like.  By the same argument as for \eqref{E:alpha lesssim diam},
$$
\alpha_j' \lesssim \rho_0 \diam(\Omega') \lesssim \rho_0 (\alpha_j')^\delta,
$$
whence $\rho_0 \geq c^{-1} (\alpha_j')^{1-\delta}$, again with a slightly smaller value of $\delta$.  

In summary, to complete the proof of Theorem~\ref{T:not extreme} (and thereby that of Theorem~\ref{T:main}) it suffices to prove the following.  
          
\begin{lemma} \label{L:reductions}
Let $\eps>0$ be sufficiently small depending on $b_0$ and $\delta>0$ be sufficiently small depending on $\eps, b_0$.  Let $\Omega \subseteq \supp a$ be a Borel set, and define $\alpha_1,\ldots,\alpha_k$ as in \eqref{E:define alphas}.  Assume that $\alpha_1 \leq \ldots \leq \alpha_k$, that
$$
\rho_0 \leq \rho(x) \leq 2\rho_0 \qtq{for all} x \in \Omega,
$$
and that 
\begin{equation} \label{E:on Omega}
\alpha_k < c, \qquad \rho_0 \geq c^{-1} \alpha_k^{1-\delta}, \qquad \diam(\Omega) \leq c \alpha_1^\delta.
\end{equation}
Then for $c$ sufficiently small, depending on the $\pi_j$, $b_0, \eps, \delta$, we have
\begin{equation} \label{E:alpha lb}
\prod_{j=1}^k \alpha_j^{b_0^j+\eps} \lesssim \rho(\Omega).
\end{equation}
\end{lemma}

We note in particular that all constants and implicit constants are independent of $\rho_0$, $\Omega$, and the $\alpha_j$.  
          
We devote the remainder of this section to the proof of Lemma~\ref{L:reductions}.  We use the method of refinements, which originated in \cite{CCC} and was further developed in similar contexts in \cite{ChRl, TW}.  

Recalling \eqref{E:def tilde rho}, 
\begin{equation} \label{E:J0 beta0}
|\partial^{\beta_0} \det D\Psi_{x_0}^{J_0}(0)| \sim \rho_0^{|b_0|_1-1} =: \lambda_0, \qquad x_0 \in \Omega.
\end{equation}

As in \cite{TW}, for $w > 0$, we say that a set $S \subseteq [-w,w]$ is a central set of width $w$ if for any interval $I \subseteq [-w,w]$, 
$$
|I \cap S| \lesssim \bigl(\tfrac{|I|}{w}\bigr)^\eps |S|.
$$

\begin{lemma} \label{L:central refinement}
For each subset $\Omega' \subseteq \Omega$ with $\rho(\Omega') \gtrsim \alpha_1^{C\eps} \rho(\Omega)$ and each $1 \leq j \leq k$, there exists a refinement $\langle \Omega' \rangle_j \subseteq \Omega'$ with $\rho(\langle \Omega' \rangle_j) \gtrsim \alpha_1^{2C\eps} \rho(\Omega')$, such that for each $x \in \langle \Omega' \rangle_j$,
\begin{equation} \label{E:def Fj}
\mathcal F_j(x,\langle \Omega' \rangle_j) \subseteq \{t : |t| \lesssim \alpha_1^\delta \ctc{and} e^{tX_j}(x) \in \langle \Omega' \rangle_j\}
\end{equation} 
is a central set whose width $w_j$ and measure satisfy
\begin{equation} \label{E:refined width}
\rho_0^{-1} \alpha_1^{2C\eps} \alpha_j \lesssim w_j \leq c \alpha_1^\delta \qtq{and} |\mathcal F_j(x,\langle \Omega' \rangle_j)| \gtrsim \rho_0^{-1} \alpha_1^{2C\eps}\alpha_j.
\end{equation}
\end{lemma}

This lemma has essentially the same proof as Lemma~8.2 of \cite{TW}, but we sketch the argument for the convenience of the reader.

\begin{proof}[Sketch proof of Lemma~\ref{L:central refinement}]
First we discard shorter-than-average $\pi_j$ fibers in $\Omega'$, leaving a subset $\Omega'' \subseteq \Omega'$ with $\rho(\Omega'') \gtrsim \rho(\Omega')$ such that for each $x \in \Omega''$,
$$
|\{t : |t|\lesssim \alpha_1^\delta, \ctc{and} e^{tX_j}(x) \in \Omega''\}| \gtrsim \tfrac{|\Omega'|}{|\pi_j(\Omega')|} \gtrsim \alpha_1^{C\eps} \rho_0^{-1} \alpha_j.
$$

Next, if $S \subseteq [-c\alpha_1^\delta,c\alpha_1^\delta]$ is a measurable set, it contains a translate $S'$ of a central set of measure at least $|S|^{1+2\eps}$ and width at most $c\alpha_1^\delta$.  Indeed, take $S' = S \cap I'$, where $I'$ is a minimal length dyadic interval with $|S \cap I'| \geq (\frac{|I'|}{\alpha_1^\delta})^\eps |S|$.  

Using the exponential map, each $\pi_j$ fiber in $\Omega''$ is naturally associated to a set $S \subseteq [-c\alpha_1^\delta,c\alpha_1^\delta]$; $S$ can be refined to a translate $S'$ of a central set; and $S'$ is then a fiber of the set $\langle \Omega' \rangle_j$.  By the definition of exponentiation, for $x \in \langle \Omega' \rangle_j$ the set $\mathcal F_j(x,\langle \Omega' \rangle_j)$ in \eqref{E:def Fj} contains 0, and it is easy to see that a 0-containing translate of a central set of width $w$ is a central set of width $2w$.  Finally, by pigeonholing, we can select only those fibers having the most popular dyadic width (there are at most $\log \alpha_1$ options).
\end{proof}

Write $J_0 = (j_1,\ldots,j_d)$.  With $\Omega_0 := \Omega$, for $1 \leq i \leq d$ we define
$$
\Omega_i := \langle \Omega_{i-1} \rangle_{j_{d-i+1}}.
$$
By Lemma~\ref{L:central refinement}, for each $i$, $\rho(\Omega_i) \gtrsim  \alpha_1^{C\eps} \rho(\Omega)$.  

Fix $x_0 \in \Omega_d$.  Let 
$$
F_1 := \mathcal F_{j_1}(x_0,\Omega_d), \qquad x_1(t) := e^{tX_{j_1}}(x_0),
$$
and for $2 \leq i \leq d$, let
\begin{gather*}
F_i := \bigl\{(t_1,\ldots,t_i) : (t_1,\ldots,t_{i-1}) \in F_{i-1}, \: t_i \in \mathcal F_{j_i}(x_{i-1}(t_1,\ldots,t_{i-1}),\Omega_{d-i+1})\bigr\}\\
x_i(t_1,\ldots,t_i) := e^{t_i X_{j_i}} x_{i-1}(t_1,\ldots,t_{i-1}).
\end{gather*}

By construction, for each $i$ and each $(t_1,\ldots,t_i) \in F_i$, 
$$
x_i(t_1,\ldots,t_i) \in \Omega_{d-i+1} \subseteq \Omega_{d-i},
$$
so $\mathcal F_{j_{i+1}}(x_i(t_1,\ldots,t_i),\Omega_{d-i})$ is a central set whose width and measure satisfy \eqref{E:refined width} (with $j_{i+1}$ in place of $j$).  Furthermore, 
\begin{equation} \label{E:Fd}
\Psi^{J_0}_{x_0}(F_d) \subseteq \Omega \qtq{and} |F_d| \gtrsim \rho_0^{-d}\alpha_1^{C\eps}\alpha^{\deg J_0};
\end{equation}
here we recall that $\deg J$ is the $k$-tuple whose $i$-th entry is the number of appearances of $i$ in the $d$-tuple $J$.  

Let $\Psi_{x_0}^N$ be the degree $N$ Taylor polynomial of $\Psi_{x_0}^{J_0}$, where $N \geq |b_0|_1+1$ is a large integer to be chosen later.  Let $Q_w = \prod_{i=1}^d [-w_i,w_i]$ and let $Q_1 = Q_{(1,\ldots,1)}$.  By scaling, the equivalence of all norms on the degree $N$ polynomials in $d$ variables, and \eqref{E:J0 beta0},
\begin{align*}
&\|\det D\Psi_{x_0}^N\|_{C^0(Q_w)} 
= \sup_{t \in Q_1} |\det D\Psi_{x_0}^N(w_1t_1,\ldots,w_dt_d)| \\
&\qquad \sim_{N,d} \sum_\beta w^\beta |\partial^\beta \det D\Psi_{x_0}^N(0)|
\geq w^{\beta_0} |\partial^{\beta_0} \det D\Psi_{x_0}^N(0)|
\sim w^{\beta_0} \lambda_0.
\end{align*}
Thus by \eqref{E:refined width}, the definition of $\lambda_0$, and some arithmetic,
\begin{equation} \label{E:det DPsi lb}
\|\det D\Psi_{x_0}^N\|_{C^0(Q_w)} \gtrsim \rho_0^{d-1} \alpha_1^{C\eps}\alpha^{\deg_{J_0}\beta_0}.
\end{equation}
(We recall that $\deg_J \beta$ is the $k$-tuple whose $i$-th entry equals $\sum_{\ell:J_\ell = i} \beta_\ell$.)

\begin{lemma} \label{L:F poly good}
If $P$ is any degree $N$ polynomial on $\R^d$, there exists a subset $F_d' \subseteq F_d$ such that $|F_d'| \gtrsim_{N,\eps,d} |F_d|$ and 
$$
|P(t)| \gtrsim_{N,\eps,d} \|P\|_{C^0(Q_w)}, \qquad t \in F_d'.
$$
\end{lemma}

The lemma follows from Lemma~6.2 of \cite{ChRl} or Lemma~7.3 of \cite{TW}.  Roughly, if $S$ is a central set of width $w_0$ and $p$ is a degree $N$ polynomial, $p$ is close to $\|p\|_{C^0([-w_0,w_0])}$ on most of $S$.  This is because the set where $p$ is small is the union of at most $N$ small intervals.  Recalling how our set $F_d$ was constructed (from a `tower' of central sets), it is possible to iterate $d$ times to obtain the lemma.  

Now we use $\Psi_{x_0}^N$ to control $\Psi_{x_0}^{J_0}$ via the following lemma, which just paraphrases Lemma~7.1 of \cite{ChRl}.  We recall that $Q_1$ is the unit cube.

\begin{lemma} \label{L:Poly like}
Let $N,C_1,c_2,c_3 > 0$.  There exists a constant $c_0>0$, depending on $C_1,c_2,c_3,N,d$, such that the following holds.  Let $\Psi:Q_1 \to \R^d$ be twice continuously differentiable and let $\Psi^N:\R^d \to \R^d$ be a degree $N$ polynomial.  Set $J_\Psi := \|\det D\Psi\|_{C^0(Q_1)}$ and assume that 
\begin{equation} \label{E:Psi PsiN}
\|\Psi\|_{C^0(Q_1)} \leq C_1, \qquad \|\Psi - \Psi^N\|_{C^2(Q_1)} \leq c_0 \scriptJ_\Psi^2.
\end{equation}
Let $G \subseteq Q_1$ be a Borel set with the property that for any degree $N^d$ polynomial $P:\R^d \to \R$,
\begin{equation} \label{G poly good}
|\{t \in G : |P(t)| \geq c_2 \|P\|_{C^0(Q_1)}\}| \geq c_3 |G|.
\end{equation}
Then
$$
|\Psi(G)| \geq c_0 |G|\|\det D\Psi^N\|_{C^0(Q_1)}.
$$
\end{lemma}

For the complete details, the reader may consult \cite{ChRl}.  We give a quick sketch of that argument here.  

\begin{proof}[Sketch proof of Lemma~\ref{L:Poly like}]
Let $P = \det D\Psi^N$ and let $G'$ denote the set on the left of \eqref{G poly good}.  By \eqref{E:Psi PsiN}, 
\begin{equation} \label{E:D's are sim}
|\det D\Psi(t)| \sim |P(t)| \sim \|P\|_{C^0(Q_1)} \sim \scriptJ_\Psi, \:\: t \in G', \qquad \|\Psi^N\|_{C^2(Q_1)} \leq 2C_1  .
\end{equation}
This first series of inequalities above imply that 
$$\int_{G'} |\det D\Psi| \geq c_0^{1/2}  |G|\|\det D\Psi^N\|_{C^0(Q_1)}.$$
It remains to show that $\Psi$ is finite-to-one on $G'$, so that $|\Psi(G')| \gtrsim \int_{G'} |\det D\Psi|$.  

First the local case.  For $c_0$ sufficiently small and $B$ any ball with radius $c_0^{1/2} \scriptJ_\Psi$ and center in $G'$, $\Psi,\Psi^N$ may be shown to be one-to-one on $10B$ and to satisfy
\begin{equation} \label{E:J sim scriptJ}
|\det D\Psi(t)| \sim |P(t)| \sim \scriptJ_\Psi, \qquad t \in 10B.
\end{equation} 
We cover $G'$ by a finitely overlapping collection of such balls $B$.

Globally, we know (it is an application of Bezout's theorem) that $\Psi^N$ is at most $C_{N,d}$-to-one on $G'$.  %In short, it is enough to bound the number of complex solutions to the system of equations $\Psi_j^N(t) = x_j$ for $x=(x_1,\ldots,x_j) \in \Psi^N(\Omega')$.  By the addition of a dummy variable $w$, we can produce homogeneous, degree $N$ polynomials $F_j:\C^{d+1} \to \C$ with $F_j(z_1,\ldots,z_d,1) = \Psi_j^N(z)-x_j$.  We may view each $F_j$ as a map on complex projective space, $F_j:\mathbb P^d \to \mathbb P$.  The varieties $\{F_j = 0\}$ are in general position since the complex Jacobian of $\Psi^N$ (which, on $\R^d$, is the square of the real Jacobian) is nonvanishing on $\Omega'$.  Thus we can bound the number of intersection points in projective space, and specializing to $w=1$ gives the bound we seek.  The result we need may now be found in \cite[p. 198]{Shafarevich}.
Thus a point $x \in \R^d$ lies in $\Psi^N(10B)$ for at most $C_{N,d}$ balls $B \in \scriptB$.  We are done if we can show that $\Psi(B) \subseteq \Psi^N(10B)$.  By the mean value theorem (applied to $(\Psi^N)^{-1}$), then Cramer's rule, \eqref{E:D's are sim}, and \eqref{E:J sim scriptJ},
$$
\dist(\Psi^N(B), (\Psi^N(10B))^c) \geq \dist(B,(10B)^c) \|(D\Psi^N)^{-1}\|_{C^0(10B)}^{-1} > c_0^{1/2} \scriptJ_\Psi \diam(B).
$$
The right side is just $c_0 \scriptJ_\Psi^2 \geq \dist(\Psi(B),\Psi^N(B))$, so we are done.  
\end{proof}

Let $D_w$ denote the dilation $D_w(t_1,\ldots,t_d) = (w_1 t_1,\ldots,w_d t_d)$.  We will apply Lemma~\ref{L:Poly like} with $\Psi = \Psi_{x_0}^{J_0} \circ D_w$, $\Psi^N = \Psi_{x_0}^N \circ D_w$, and $G = D_w F_d$.  By Lemma~\ref{L:F poly good}, we just need to verify \eqref{E:Psi PsiN}.  

Since $w_j \leq 1$ for each $j$, $\|\Psi\|_{C^2(Q_1)} \leq \|\Psi_{x_0}^{J_0}\|_{C^2(Q_w)} \lesssim 1$.  For the error bound, \begin{equation}\label{E:poly approx}
\|\Psi_{x_0}^{J_0} - \Psi_{x_0}^N\|_{C^2(Q_w)} \lesssim \max_i w_i^{N-1} \|\Psi_{x_0}^{J_0}\|_{C^{N+1}(Q_w)} \lesssim (c \alpha_1^\delta)^N,
\end{equation}
where $c$ is as in \eqref{E:on Omega}. (Recall that implicit constants do not depend on $c$.)  We choose $N$ larger than $\delta^{-1}(10 \deg_{J_0} \beta_0 + 10d)$, and then choose $c$ sufficiently small.  Combining \eqref{E:poly approx}, \eqref{E:on Omega}, and \eqref{E:det DPsi lb}, 
$$
\|\Psi_{x_0}^{J_0} - \Psi_{x_0}^N\|_{C^2(Q_w)} \leq c_0 (\prod_j w_j)^2 \|\det D\Psi_{x_0}^N\|_{C^0(Q_w)}^2.
$$
For $c_0$ sufficiently small, this implies that 
$$\|\det D\Psi_{x_0}^{J_0} - \det D \Psi_{x_0}^N\|_{C^0(Q_w)} < \tfrac12  \|\det D\Psi_{x_0}^N\|_{C^0(Q_w)},$$
so $\|\det D\Psi_{x_0}^{J_0}\|_{C^0(Q_w)} \geq \tfrac12 \|\det D\Psi_{x_0}^N\|_{C^0(Q_w)}$.  Rescaling gives us \eqref{E:Psi PsiN}.  

Applying Lemma~\ref{L:Poly like}, inequality \eqref{E:det DPsi lb}, and $b_0 = \deg J_0 + \deg_{J_0}\beta_0$, 
$$
| \Omega| \geq |\Psi_{x_0}^{J_0}(F_d)| \gtrsim |F_d| \rho_0^{d-1} \alpha_1^{C\eps}\alpha^{\deg_{J_0}\beta_0} \gtrsim \rho_0^{-1} \alpha_1^{2C\eps} \alpha^{b_0}.
$$

The proof of Theorem~\ref{T:main} is finally complete.  

%%%%%%%%%%%%%%%%%%%%%%%%%%%%%%%%%%%%%%
%%%%%%%%%%%%%%%%%%%%%%%%%%%%%%%%%%%%%%
%%%%%%%%%%%%%%%%%%%%%%%%%%%%%%%%%%%%%%

\section{Appendix:  The proof of Proposition~\ref{P:polytope prop}}

%%%%%%%%%%%%%%%%%%%%%%%%%%%%%%%%%%%%%%
%%%%%%%%%%%%%%%%%%%%%%%%%%%%%%%%%%%%%%
%%%%%%%%%%%%%%%%%%%%%%%%%%%%%%%%%%%%%%

In this section we prove Proposition~\ref{P:polytope prop}, which was used in proving Propositions~\ref{P:optimality} and~\ref{P:equivalence}.  We fix, for the remainder of this section, a point $b_0 \in [0,\infty)^k$.  An object is admissible if it may be chosen from a finite collection, depending only on $b_0$, of such objects, and all implicit constants will be admissible (i.e.\ depending only on $b_0$).  

The following two lemmas show that conclusions (i) and (ii) of Proposition~\ref{P:polytope prop} are equivalent.  

\begin{lemma}\label{L:A to v0} 
If $\scriptA \subseteq \Z_0^k$ is a finite set and $b_0 \notin \scriptP(\scriptA)$, there exist $\eps > 0$ and $v_0 \in (\eps,1]^k$ such that $v_0 \cdot b_0 + \eps < v_0 \cdot p$ for every $p \in \scriptP(\scriptA )$.  
\end{lemma}

\begin{lemma} \label{L:v0 to A}
If $v_0 \in (0,1]^k$, there exists a finite set $\scriptA \subseteq \Z_0^k$ such that $b_0 \notin \scriptP(\scriptA)$ and 
$$
\{b \in \Z_0^k : v_0 \cdot b_0 < v_0 \cdot b\} \subseteq \scriptP(\scriptA).
$$
\end{lemma}

\begin{proof}[Proof of Lemma~\ref{L:A to v0}]
We may assume that $b_0 \neq (0,\ldots,0)$ and $\scriptA \neq \emptyset$; otherwise, the result is trivial.  Since $b_0 \notin \scriptP(\scriptA)$, there exists $v_1 \in \R^k$ such that $v_1 \cdot b_0 < v_1 \cdot p$ for every $p \in \scriptP(\scriptA)$.  Since $\scriptP(\scriptA)$ contains a translate of $[0,\infty)^k$, $v_1 \in [0,\infty)^k$.  We may assume that $v_1 \in [0,1]^k$.  Let 
$$
\delta := \tfrac12|b_0|_1^{-1} \min_{b \in \scriptA} v_1 \cdot (b-b_0).
$$
Since $\scriptA$ is finite, $\delta >0$.  Let $v_2 := v_1 + (\delta,\ldots,\delta)$.  Then $v_2 \in [\delta,1+\delta]^k$.  If $b \in \scriptA$,
$$
b \cdot v_2 = v_1 \cdot b_0 + v_1 \cdot (b-b_0) + \delta |b|_1 \geq v_2 \cdot b_0 + \delta |b_0|_1\geq v_2 \cdot b_0 + \delta.
$$
The conclusion thus holds with $\eps := \frac12\frac\delta{1+\delta}$, $v_0 := \frac{v_2}{1+\delta}$.
\end{proof}

\begin{proof}[Proof of Lemma~\ref{L:v0 to A}]
Let $\eps := \min_i v_0^i$ and let $N := \lceil k\eps^{-1}(b_0 \cdot v_0 + 1)\rceil$.  If $p \in \Z_0^k$ and $|p|_1 \geq N$, 
$$
v_0 \cdot p \geq \min_j v_0^j \max_i p^i \geq \eps (\tfrac Nk) \geq b_0 \cdot v_0 + 1,
$$
so the conclusion holds with 
$$
\scriptA := \{b \in \Z_0^k : |b|_1 \leq N \ctc{and} v_0 \cdot b > v_0 \cdot b_0\}.
$$
\end{proof}

The following lemma implies that the conclusions of Proposition~\ref{P:polytope prop} hold whenever $\scriptB$ is a finite set with $\#\scriptB\leq k+1$.

\begin{lemma} \label{L:simplex extreme}
Let $\scriptB \subseteq \Z_0^k$ be a finite set.  Assume that $\#\scriptB \leq k+1$ and that $b_0 \notin \scriptP(\scriptB)$.  Then there exist admissible $\eps > 0$ and $v_0 \in (\eps,1]^k$ such that $b \cdot v_0 > b_0 \cdot v_0 + \eps$ for every $p \in \scriptP(\scriptB)$.  
\end{lemma}

The same proof shows that for any finite $\scriptB$ with $b_0 \notin \scriptP(\scriptB)$, there exist $\eps>0$ and $v_0 \in (\eps,1]^k$, taken from a finite list that depends only on $b_0$ and $m$, such that $b \cdot v_0 > b_0 \cdot v_0 + \eps$ for every $p \in \scriptP(\scriptB)$, but for simplicity, we only prove the version that we use.  

\begin{proof}
The conclusion is trivial if $\scriptB = \emptyset$, so we write $\scriptB = \{b_1,\ldots,b_m\}$ with $m \leq k+1$.  By Lemma~\ref{L:A to v0}, the conclusion is trivial if $\{b_1,\ldots,b_m\}$ is admissible; we will reduce to this case.  

If $|b_i|_1 > |b_0|_1$, $1 \leq i \leq m$, the conclusion holds with $v_0 = (1,\ldots,1)$, $\eps = \tfrac12(\lceil |b_0|_1+1 \rceil-1)$.  Reindexing if necessary, we may assume that $|b_1|_1 \leq |b_0|_1$, in which case $\{b_1\}$ is admissible.  

Assume that for some $j < m$, $\{b_1,\ldots,b_j\}$ is admissible.  By assumption, $b_0 \notin \scriptP(\{b_1,\ldots,b_j\})$, so by Lemma~\ref{L:A to v0}, there exist admissible $\eps_j>0$, $v_j \in (\eps_j,1]^k$ such that $v_j \cdot b_0 + \eps_j < v_j \cdot b_i$ for $1 \leq i \leq j$.  If $v_j \cdot b_0 +\eps_j < v_j \cdot b_i$ for every $i$, the conclusion of the lemma holds with $\eps = \eps_j$, $v_0= v_j$.  Otherwise, after reindexing, we may assume that $v_j \cdot b_{j+1} \leq v_j \cdot b_0$.  Therefore $b_{j+1}$ is admissible, and hence $\{b_1,\ldots,b_{j+1}\}$ is admissible as well.  The procedure must terminate after at most $m$ ($\leq k+1$) steps, and so the lemma is proved.
\end{proof}

Lemma~\ref{L:simplex extreme} has the following corollary.

\begin{lemma} \label{L:simplex separation}
Under the hypotheses of Lemma~\ref{L:simplex extreme}, there exists an admissible $\eps > 0$ such that if 
$$
b(\theta) := \sum_{i=1}^m \theta_i b_i
$$
is any convex combination of $b_1,\ldots,b_m$, there exists an $i$, $1 \leq i \leq k$ such that $b^i(\theta) \geq b_0^i+\eps$.
\end{lemma}

\begin{proof}
By Lemma~\ref{L:simplex extreme}, there exist admissible $\eps>0$, $v_0 \in (\eps,1]^k$ such that 
$$
\eps < (b(\theta)-b_0) \cdot v_0 \leq (\sum_{i=1}^k v_0^i) \max_{1 \leq i \leq k} (b^i(\theta)-b_0^i) \leq \max_{1 \leq i \leq k} (b^i(\theta)-b_0^i).
$$
\end{proof}

Finally, we are ready to complete the proof of Proposition~\ref{P:polytope prop}.

\begin{proof}[Proof of Proposition~\ref{P:polytope prop}]
Let $C>|b_0|_1$ be a large constant, to be determined (admissibly) in a moment.  Define $\mathcal A := \scriptB' \cup \scriptB''$, where 
\begin{align*}
\scriptB' &:= \{b \in \scriptB : |b|_1 \leq C\}\\
\scriptB'' &:= \{Ce_i:1 \leq i \leq k\}.
\end{align*}
Here $e_i$ denotes the $i$-th standard basis vector.  Then since $\scriptP(\scriptB'') = \scriptP(\{b \in \Z_0^k : |b|_1 \geq C\}$, $\scriptP(\scriptB) \subseteq \scriptP(\scriptA)$.  It remains to show that for $C$ sufficiently large, $b_0\notin\scriptP(\mathcal A)$.  

Assume that $b_0 \in \scriptP(\scriptA)$.  By Carath\'eodory's Theorem from combinatorics (see, for instance, \cite[p. 46]{ZieglerPolytopes}), $b_0 \succeq \sum_{l=1}^{k+1} \theta_l a_l$, for some $a_1,\ldots,a_{k+1} \in \scriptA$ and $0 \leq \theta_l \leq 1$ satisfying $\sum_l \theta_l = 1$.  Reindexing if necessary,
\begin{equation} \label{E:b0 geq}
b_0 \succeq \sum_{l=1}^j \theta_l Ce_{i_l} + \sum_{l=j+1}^{k+1} \theta_l b_l,
\end{equation}
where $b_{j+1},\ldots,b_{k+1} \in \scriptB'$.  Since $C>|b_0|_1$, $\sum_{l=j+1}^{k+1} \theta_l > 0$, and since $b_0 \notin \scriptP(\scriptB')\subseteq \scriptP(\scriptB)$, $\sum_{l=1}^j \theta_l > 0$.  

Let 
$$
b(\theta) := (\sum_{l=j+1}^{k+1} \theta_l)^{-1} \sum_{l=j+1}^{k+1} \theta_l b_l.
$$
By Lemma~\ref{L:simplex separation}, there exists an $i$, $1 \leq i \leq k+1$ such that $b^i(\theta) \geq b_0^i + \eps$, where $\eps > 0$ depends only on $b_0$ (crucially, not on $C$).  By \eqref{E:b0 geq}, 
$$
b_0 \succeq (\sum_{l=j+1}^{k+1} \theta_j)b(\theta),
$$
so comparing the $i$-th coordinates, we see that 
$$
\sum_{l=j+1}^{k+1} \theta_j \leq \frac{b_0^i}{b_0^i+\eps} \leq \frac{|b_0|_\infty}{|b_0|_\infty+\eps},  
$$
so 
\begin{equation} \label{E:theta lb}
\sum_{l=1}^j \theta_j \geq 1-\frac{|b_0|_\infty}{|b_0|_\infty+\eps} = \frac{\eps}{|b_0|_\infty+\eps}.
\end{equation}
On the other hand, by \eqref{E:b0 geq} and the fact that all coordinates of the $b_i$ are nonnegative, $\sum_{l=1}^j \theta_j \leq \frac{|b_0|_1}C$.  For $C=C(\eps,b_0)$ sufficiently large (admissible since $\eps$ is), this contradicts \eqref{E:theta lb}, and the proof of Proposition~\ref{P:polytope prop} is complete.
\end{proof}

%%%%%%%%%%%%%%%%%%%%%%%%%%%%%%%%%%%%%%%%%%%%%%%
%%%%%%%%%%%%%%%%%%%%%%%%%%%%%%%%%%%%%%%%%%%%%%%
%%%%%%%%%%%%%%%%%%%%%%%%%%%%%%%%%%%%%%%%%%%%%%%

%%%%%%%%%%%%%%%%%%%%%%%%%%%%%%%%%%%%%%%%%%%%%%%
%%%%%%%%%%%%%%%%%%%%%%%%%%%%%%%%%%%%%%%%%%%%%%%
%%%%%%%%%%%%%%%%%%%%%%%%%%%%%%%%%%%%%%%%%%%%%%%

\end{document}